\documentclass[12pt]{amsart}
\usepackage{amsmath,amssymb, xypic,verbatim,amscd,color}
\usepackage{graphicx}
\usepackage{colordvi}

\numberwithin{equation}{section}

\newcommand\dlog{\operatorname{dlog}}

\newcommand\frg{\mathfrak{g}}
\newcommand\frh{\mathfrak{h}}

\newcommand\home{\operatorname{Hom}}

\newcommand\pone{\Bbb{P}^1}

\newcommand\mv{\mathcal{V}}
\newcommand\tensor{\otimes}

\newcommand\mh{\mathcal{H}}

\newcommand\Res{\operatorname{Res}}

\newcommand{\leto}[1]{\stackrel{#1}{\to}}

\newtheorem{theorem}{Theorem}[section]
\newtheorem{remark}[theorem]{ Remark}

\newtheorem{question}[theorem]{Question}
\newtheorem{proposition}[theorem]{Proposition}

\newtheorem{lemma}[theorem]{Lemma}

\newtheorem{definition}[theorem]{Definition}
\newtheorem{defi}[theorem]{Definition}

\begin{document}
\title[Conformal blocks and cohomology]{Conformal blocks and cohomology in genus 0}
\author{Prakash Belkale}
\author{Swarnava Mukhopadhyay}
\thanks{The authors were partially supported by NSF grant  DMS-0901249.}

\address{Department of Mathematics\\ UNC-Chapel Hill\\ CB \#3250, Phillips Hall
\\ Chapel Hill, NC 27599}
\email{belkale@email.unc.edu}
\email{swarnava@umd.edu}
\subjclass{Primary 17B67, 14H60, Secondary 32G34, 81T40}
\begin{abstract} We give a characterization of conformal blocks in terms of the singular cohomology of suitable smooth projective varieties, in genus $0$ for classical Lie algebras and $G_2$.
\end{abstract}
\maketitle
\section{Introduction}
Consider  a finite dimensional simple Lie algebra $\frg$, a non-negative integer $k$ called the level and a $N$-tuple $\vec{\lambda}=(\lambda_1,\dots,\lambda_N)$ of
dominant weights of $\frg$ of level $k$. Associated to this data there is a vector bundle of conformal blocks $\mathcal{V}=\mathcal{V}_{\vec{\lambda},k}$ on $\overline{\mathfrak{M}}_{g,N}$, the moduli stack of stable $N$-pointed curves of genus $g$ ~\cite{TK,TUY}. The fibers of $\mv$ on $\mathfrak{M}_{g,N}$  can also be described in terms of sections of natural line bundles on suitable moduli stacks of parabolic principal bundles on $N$-pointed curves of genus $g$ (see the survey ~\cite{sorger}).

Now suppose $g=0$. Let $\mathcal{C}$ be the configuration space of $N$ distinct points on $\Bbb{A}^1$.
Let $\vec{z}=(z_1,\dots,z_N)\in\mathcal{C}$ and  $\mathfrak{X}(\vec{z})$ be the corresponding $N$-pointed curve. Consider the space of conformal blocks $V^{\dagger}_{\vec{\lambda}}(\mathfrak{X}(\vec{z}))$ associated to this data. In ~\cite{B} (generalizing work of Ramadas \cite{TRR}, and using work of Schechtman-Varchenko ~\cite{SV}), an injective map from $V^{\dagger}_{\vec{\lambda}}(\mathfrak{X}(\vec{z}))$
to the (topological) cohomology of a smooth and projective variety $\overline{Y}_{\vec{z}}$, consistent with connections, was constructed (we recall this construction in Section ~\ref{exam}). Our aim here is to characterize
the image of this injective map, for classical $\frg$ and $G_2$. This gives a cohomological description of genus $0$ conformal blocks. We hope that the result extends to the remaining cases for $\frg$, but note that our methods get more difficult to implement in these cases (see Remark ~\ref{f4remark}).

 Let $\mathfrak{h}$ be a Cartan subalgebra of $\frg$ and $\Delta \subset \mathfrak{h}^*$ be the root system of $(\frg,\mathfrak{h})$. Associated to $\Delta$, the Lie algebra $\frg$ has a the following Cartan decomposition:
$$\frg=\mathfrak{h}\bigoplus_{\alpha \in \Delta}\frg_{\alpha},$$ where $\frg_{\alpha}$ is a one dimensional vector space of weight $\alpha$. The set of roots $\Delta$, is decomposed into a union $\Delta_{+}\cup \Delta_{-}$ of positive and negative roots. The set of simple (positive) roots is denoted by $R=\{\alpha_1,\dots, \alpha_r\}$, where $r$ is the rank of the Lie algebra $\frg$. Let $(\ ,\ )$ denote the Cartan Killing form normalized such that $(\theta, \theta)=2$, where $\theta$ is the highest root. We identify $\mathfrak{h}$ with $\mathfrak{h}^*$ using $(\ ,\ )$.

Let us recall the main result of ~\cite{B}. Assume that $\mu=\sum_{i=1}^{N} \lambda_i$ is in the root lattice (otherwise $V^{\dagger}_{\vec{\lambda}}(\mathfrak{X})=0$) and $(\lambda_i,\theta)\leq k, i\in [N]$. Write
$\mu=\sum_{p=1}^r n_p \alpha_p$, where $\alpha_p$ are the simple positive roots and $n_p\geq 0$. Fix a  map $\beta:[M]=\{1,\dots,M\} \to R$, so that
$\mu=\sum_{a=1}^M \beta(a)$.

Introduce variables $t_1,\dots,t_M\in \pone-\{\infty,z_1,\dots,z_N\}$. We will consider the variable $t_a$ to be colored by the simple root $\beta(a)$. Now consider the following Schechtman-Varchenko master function  ~\cite{SV}:
$$\mathcal{R}=\displaystyle\prod_{1\leq i<j\leq N}(z_i-z_j)^{\frac{-(\lambda_i,\lambda_j)}{\kappa}}\displaystyle\prod_{a=1}^M \displaystyle\prod_{j=1}^N(t_a-z_j)^{\frac{(\lambda_j,\beta(a))}{\kappa}}
\displaystyle\prod_{1\leq a< b\leq M}  (t_a-t_b)^{\frac{-(\beta(a),\beta(b))}{\kappa}},$$
where $\kappa = k+g^*$ where $g^*$ is the dual Coxeter number of $\frg$. We observe that the exponents in the Schechtman-Varchenko master function  $\mathcal{R}$ may not be integers. Fix a sufficiently divisible positive integer $C$ so that
\begin{equation}\label{snl1}
C(\lambda_i,\lambda_j),\ C(\beta(a),\beta(b)),\ C(\beta(a),\lambda_i)\in \Bbb{Z}, \forall a,b\in [M],\ i,j\in [N],\ a<b,\ i<j,
\end{equation}
and an ``evenness'' condition
\begin{equation}\label{snl2}
C(\alpha,\alpha) \in 2\Bbb{Z}, \forall \alpha\in R.
\end{equation}


Consider the following complement of a
hyperplane arrangement $$X_{\vec{z}}=\{(t_1,\dots, t_M)\in \Bbb{A}^M: t_a\neq t_b,  t_a \neq z_i, i\in [N], a<b\in [M]\},$$
and an unramified (possibly disconnected) cover of $X_{\vec{z}}$ given by $Y_{\vec{z}}=\{(t_1,\dots, t_M,y)\mid y^{C\kappa} = P\}$,
where
\begin{equation}
P=\displaystyle\prod_{1\leq i<j\leq N}(z_i-z_j)^{-C(\lambda_i,\lambda_j)}\displaystyle\prod_{a=1}^M \displaystyle\prod_{j=1}^N(t_a-z_j)^{C(\lambda_j,\beta(a))}
\displaystyle\prod_{1\leq a< b\leq M}  (t_a-t_b)^{-C(\beta(a),\beta(b))}.
\end{equation}
The master function $\mathcal{R}$ lifts to a single valued function of $Y_{\vec{z}}$ which we again denote by $\mathcal{R}$. The group $\mu_{C\kappa}\subseteq \Bbb{C}^*$, of $(C\kappa)$-th roots of unity, acts on ${Y}_{\vec{z}}$. Let $\Sigma$ be the subgroup of the symmetric group $S_M$ on $M$ letters given by
$$\Sigma=\{\sigma\in S_M\mid \beta(\sigma(a))=\beta(a)\}.$$ Clearly, $\Sigma$ is a product of symmetric groups $S_{n_1}\times \dots \times S_{n_r}$, where $S_{n_p}=\{\operatorname{id}\}$ for $n_p=0$.

The ``evenness" condition \eqref{snl2} ensures that the permutation $(a,b)$ acts trivially on the function $(t_a-t_b)^{-C(\beta(a),\beta(b))}$ when $\beta(a)=\beta(b)$, and hence, $G=\Sigma\times \mu_{C\kappa}$ acts on ${Y}_{\vec{z}}$ by the rule
$$(\sigma,c)(t_1,\dots,t_M,y)=(t_{\sigma^{-1}(1)},\dots,t_{\sigma^{-1}(M)},cy).$$

Define the character $\chi:G\to \Bbb{C}^*$ by
$$\chi(\sigma,c)=c^{-1} \epsilon(\sigma),$$
where $\epsilon$ is the sign character.

By equivariant resolution of singularities, the action of $G$ on ${Y}_{\vec{z}}$ extends to a suitable smooth compactification $\overline{Y}_{\vec{z}}$. In ~\cite{TRR} (for $\frg=\mathfrak{sl}_2$) and subsequently in ~\cite{B} (for arbitrary $\frg$) a natural inclusion (which by ~\cite{SV} preserves connections)
\begin{equation}\label{brahms}
V^{\dagger}_{\vec{\lambda}}(\mathfrak{X}(\vec{z}))\hookrightarrow H^{M,0}(\overline{Y}_{\vec{z}},\Bbb{C})
\end{equation}
 was constructed.  Recall that $H^{M,0}(\overline{Y}_{\vec{z}},\Bbb{C})$ injects into $H^{M}({Y}_{\vec{z}},\Bbb{C})$ and is independent of the compactification. The construction in ~\cite{B}, shows that the image of ~\eqref{brahms} is in $H^{M,0}(\overline{Y}_{\vec{z}},\Bbb{C})^{\chi}$ (this was noted before for $\frg=\mathfrak{sl}_2$ in ~\cite{TRR}). This uses the fact that the relevant correlations functions (see Section ~\ref{marcel}) are symmetric in variables $t_a$ and $t_b$ if $\beta(a)=\beta(b)$ (the sign character is because differentials $dt_a$ and $dt_b$ skew commute in the exterior algebra of forms).

\begin{theorem}\label{main}
The inclusion
\begin{equation}\label{brahms2}
V^{\dagger}_{\vec{\lambda}}(\mathfrak{X}(\vec{z}))\hookrightarrow (H^{M,0}(\overline{Y}_{\vec{z}},\Bbb{C}))^\chi,
\end{equation}
 is an isomorphism for $\frg$ classical and $G_2$.
\end{theorem}
Theorem ~\ref{main} for $\frg=\mathfrak{sl}(2)$ is due to Looijenga ~\cite{loo} and Varchenko (unpublished).
\begin{remark} If $H$ is the weight $M$-part of the cohomology group $H^M(Y_{\vec{z}},\Bbb{C})$ then $H$ carries an action of $\Sigma\times \mu_{C\kappa}$ (using functoriality of mixed Hodge structures). The group on the right hand side of ~\eqref{brahms2} is ($H^{M,0})^{\chi}$.
\end{remark}

\begin{question}
Is there a generalization of Theorem ~\eqref{brahms2} in higher genus (consistent with connections)? See ~\cite{BFS}, Chapter 6, Section 19.9 for a related statement.
\end{question}

A few reductions can be made immediately. Consider an element $\omega\in (H^{M,0}(\overline{Y}_{\vec{z}},\Bbb{C}))^\chi$. On ${Y}_{\vec{z}}$, $\omega$ can be expressed as a differential form $\mathcal{R}q^*\Omega$ where
 $q:Y_{\vec{z}}\to X_{\vec{z}}$ is the covering map. Furthermore $\Omega$ is of the form
 $$\Omega=Q(t_1,\dots,t_M) dt_1\wedge dt_2\wedge\dots\wedge dt_M,$$ where $Q$ is symmetric under the action of $\Sigma$. These properties follow immediately from the invariance conditions.
 The main body of the proof is broken up into  two steps:
 \subsubsection{The first step}
 With $\Omega$ as above, using the symmetry of $Q$ under the action of $\Sigma$ and the fact that $\omega=\mathcal{R}q^*\Omega$  extends to (any) compactification of  $Y_{\vec{z}}$ (or equivalently that $\mathcal{R}\Omega$ is  a multivalued,  square integrable form on $X_{\vec{z}}$), we will show that $\Omega$ is a log-form on $X_{\vec{z}}$ (the notion  of a log-form
 is reviewed in Section ~\ref{starlog}). A part of this argument is done case by case (for classical Lie algebras and $G_2$).
 \subsubsection{The second step}
 We will use results in ~\cite{SV}, ~\cite{bea} and ~\cite{FSV} to conclude the argument. In ~\cite{SV}, elements in duals of tensor products of Verma modules (of the corresponding Lie algebra ``without Serre relations'')  are constructed from suitable log forms on weighted hyperplane arrangements. We show that these elements lie in the space of conformal blocks thereby showing the surjectivity of ~\eqref{brahms2}. This step again uses the square integrability of $\mathcal{R}\Omega$. It also uses the description of the space of conformal blocks as a quotient of coinvariants (cf. ~\cite{FSV}, ~\cite{bea}). This step should be compared with \cite{loo2}.

\section{Conformal blocks}\label{exam}
The affine Lie algebra $\hat{\frg}$ is defined to be
$$\hat{\frg}=\frg\tensor \Bbb{C}((\xi))\oplus \Bbb{C}c,$$
where $c$ is an element in the center of $\hat{\frg}$ and the Lie algebra structure is defined by
$$[X\tensor f(\xi),Y\tensor g(\xi)]= [X,Y]\tensor f(\xi)g(\xi) + (X,Y)\Res_{\xi=0}(gdf).c,$$
where $f,g\in \Bbb{C}((\xi))$ and $X,Y\in \frg$.

Let $$X(n)= X\tensor \xi^n, \ X=X(0)=X\tensor 1,\ X\in \frg, n\in \Bbb{Z}.$$

\subsection{Representation theory of affine Lie-algebras}\label{UA}
 Recall that finite dimensional irreducible representations of $\frg$ are parameterized by the set of dominant integral weights $P_+$  considered a subset of $\frh^*$. To $\lambda\in P_+$, the corresponding irreducible representation $V_{\lambda}$ contains a non-zero vector $|\lambda\rangle\in V_{\lambda}$ (the highest weight vector) such that
 $$H|\lambda\rangle=\lambda(H)|\lambda\rangle, H\in \frh,$$ $$X_{\alpha}|\lambda\rangle=0, X_{\alpha}\in \frg_{\alpha}, \forall \alpha\in \Delta_{+}.$$

 We will fix a level $k$ in the sequel. Let $P_k=\{\lambda\in P_+\mid (\lambda,\theta)\leq k\}$ denote the set of dominant integral weights of level $k$, where  $\theta$ is the highest (longest positive) root.

For each $\lambda\in P_k$ there is a corresponding highest weight integrable  irreducible $\hat{\frg}$-module $\mathcal{H}_{\lambda}\supseteq V_{\lambda}$ (see ~\cite{B} for more details). The representation $\mathcal{H}_{\lambda}$ when $\lambda=0$ (still at level $k$) is called the vacuum representation at level $k$.

\subsection{Conformal blocks}
We will work in genus $g=0$, but state the definitions in greater generality.
To define conformal blocks  we will fix a stable $N$ pointed curve with formal coordinates
$\mathfrak{X}= (C; P_1,\dots,p_N,\eta_1,\dots,\eta_N)$ with $\eta_i: \hat{\mathcal{O}}_{C,P_i}\leto{\sim}
\Bbb{C}[[\xi_i]],\ i=1,\dots, N$, and choose
$\vec{\lambda}=(\lambda_1,\dots,\lambda_N)\in P_k^N$. There are a
number of definitions relevant to the situation: Let
$$\hat{\frg}_N=\bigoplus_{i=1}^N \frg\tensor_{\Bbb{C}}
\Bbb{C}((\xi_i))\oplus \Bbb{C}c$$ be the Lie algebra with $c$ a
central element and the Lie bracket given by
$$[\sum_{i=1}^N X_i\tensor f_i,\sum_{i=1}^N Y_i\tensor g_i]= \sum_{i=1}^N [X_i,Y_i]\tensor f_ig_i +
\sum_{i=1}^N (X_i,Y_i) \operatorname{Res_{P_i}}(g_i df_i)c.$$

Using the chosen formal coordinate we can realize  $$\frg(\mathfrak{X})=\frg\tensor_{\Bbb{C}}H^0(C-\{P_1,\dots,P_N\},\mathcal{O})$$ as a Lie subalgebra of $\hat{\frg}_N$. Let $\vec{\lambda}$ be as above. Set
$$\mh_{\vec{\lambda}}=\mh_{\lambda_1}\tensor\dots\tensor\mh_{\lambda_N}.$$

For a given $X\in \frg$ and $f\in\Bbb{C}((\xi_i))$, define
$\rho_i(X\tensor f)$ an endomorphism of $\mh_{\vec{\lambda}}$ by
$$\rho_i(X\tensor f)| v_1\rangle\tensor\dots\tensor |v_N\rangle=|
v_1\rangle\tensor\dots \tensor(X\tensor f|v_i\rangle) \tensor\dots\tensor |v_N\rangle,$$
where  $| v_i\rangle\in\mh_{\lambda_i}$ for each $i$.

We can now define the action of $\hat{\frg}_N$ on
$\mh_{\vec{\lambda}}$ by
$$(X_1\tensor f_1,\dots, X_N\tensor f_N)| v_1\rangle\tensor\dots\tensor
|v_N\rangle = \sum_{i=1}^N \rho_i(X_i\tensor f_i)| v_1\rangle\tensor\dots\tensor
|v_N\rangle.$$
\begin{defi}
Define the space of conformal blocks  ~\cite{Ueno}
$$V^{\dagger}_{\vec{\lambda}}(\mathfrak{X})=
\home_{\Bbb{C}}(\mh_{\vec{\lambda}}/\frg(\mathfrak{X})\mh_{\vec{\lambda}},\Bbb{C}).$$

\end{defi}

Elements of
$V^{\dagger}_{\vec{\lambda}}(\mathfrak{X})$   are
frequently denoted by $\langle\Psi|$, those
of $\mh_{\vec{\lambda}}$ by
$|\Phi\rangle$, and the pairing by $\langle\Psi|\Phi\rangle$.

\subsection{Propagation of vacua}
Add a new point $P_{N+1}$ together with the vacuum representation
$V_0$ of level $k$, at $P_{N+1}$. Also fix a formal neighborhood at
$P_{N+1}$. We therefore have a new pointed curve $\mathfrak{X}'$,
and an extended
$\vec{\lambda}'=(\lambda_1,\dots,\lambda_N,\lambda_{N+1}=0)$. The
propagation of vacuum gives an  isomorphism \cite{Ueno}:
$$V^{\dagger}_{\vec{\lambda}'}(\mathfrak{X}')\leto{\sim}V^{\dagger}_{\vec{\lambda}}(\mathfrak{X}),
\ \langle\widehat{{\Psi}}|\mapsto \langle{\Psi}|$$ with the key formula
$$\langle\widehat{{\Psi}}|(|\Phi\rangle\tensor|0\rangle)= \langle{\Psi}|\Phi\rangle.$$

\section{Correlation functions and the extension theorem}\label{corro}
\subsection{Correlation functions}
 Suppose $\mathfrak{X}$ be a stable $N$-pointed curve with formal coordinates. Let $\langle{\Psi}|\in V^{\dagger}_{\vec{\lambda}}(\mathfrak{X})$,
 $|\Phi\rangle\in \mh_{\vec{\lambda}}$, $Q_1,\dots,Q_M\in C\setminus\{P_1,\dots,P_N\}$, $Q_a\neq Q_b, a<b$ and corresponding elements
 $X_1,\dots,X_M\in \frg$. There is a very important differential in  $\bigotimes_{a=1}^M\Omega^1_{C,Q_a}$ called a correlation function
 \begin{equation}\label{correlator}
\Omega= \langle{\Psi}|X_1(Q_1)X_2(Q_2)\dots X_M(Q_M)|\Phi\rangle.
  \end{equation}
 Here $\Omega^1_{C}$ is the vector bundle of holomorphic one-forms on $C$.

We now briefly recall the definition and some important properties of correlation functions (cf. ~\cite{Ueno} for more details). One way to define \eqref{correlator} is via propagation by vacua: add points $Q_1,\dots Q_M$ with formal coordinates $\xi_1,\dots,\xi_M$ and consider the elements $X_a(-1)|0\rangle$ in the vacuum representation at those points. Let $\vec{0}_M$ denote an $M$-tuple of $0$-weights and let $\vec{\lambda}'=(\vec{0}_M,\vec{\lambda})$. Consider the element  $|X_1(-1)|0\rangle\tensor X_2(-1)|0 \rangle\dots X_M(-1)|0\rangle\tensor |\Phi\rangle$ of $\mathcal{H}_{\vec{\lambda}'}$. Let $\langle \widehat{\Psi}|\in V^{\dagger}_{\vec{\lambda}'}(\mathfrak{X}')$ denote the image, under propagation of vacua, of $\langle \Psi |$. Here $\mathfrak{X}'$ is the marked curve associated to $\pone$ and the points $(P_1,\dots,P_N,Q_1,\dots,Q_M)$ and corresponding formal coordinates.
\begin{definition}
The correlation function \eqref{correlator}
associated to $|\Phi\rangle$, $\langle \Psi |$, $X_1,\dots, X_M$, and distinct points $Q_1,\dots, Q_M\in C\setminus\{P_1,\dots,P_N\}$on a smooth curve $C$ is defined to be the differential form
$$ \langle \widehat{\Psi}|X_1(-1)|0\rangle\tensor X_2(-1)|0 \rangle\tensor\dots \tensor X_M(-1)|0\rangle\tensor |\Phi\rangle d\xi_1\dots d\xi_M.$$
It is shown in ~\cite{Ueno} that the above definition is independent of the chosen coordinates $\xi_1\dots, \xi_M$, and defines an element of $\bigotimes_{a=1}^M\Omega^1_{C,Q_a}$.
\end{definition}
The correlation function \eqref{correlator} has the following properties (cf. Page 70 of ~\cite{Ueno}).
\begin{enumerate}
\item $\Omega$ is linear with respect to $|\Phi\rangle $ and multi-linear in $X_a$'s.
\item If $(X_a,X_b)=0$ where $1\leq a < b\leq M$, then the form $\Omega$ has at most simple poles along the diagonal $Q_a=Q_b$.
\item  The form $\Omega$ has at most simple poles along diagonals of the form $Q_a=P_i$ for all $1\leq a \leq M$ and $1\leq i \leq N$.
\end{enumerate}
\subsection{The extension theorem}\label{marcel}
We will henceforth consider the case $C=\pone$, with a chosen $\infty$ and a coordinate $z$ on $\Bbb{A}^1=\pone-\{\infty\}$.
Consider distinct points $P_1,\dots,P_N\in \Bbb{A}^1\subset \pone$ with $z$-coordinates $z_1,\dots,z_N$ respectively. The standard coordinate $z$ endows each $P_i$ with a formal coordinate. Let $\mathfrak{X}$ be the resulting $N$-pointed curve with formal coordinates.

\begin{definition}\label{gold}
For every positive root $\delta$, make a choice of a non-zero element $f_{\delta}$
in the root space $\frg_{-\delta}$.
\end{definition}

Assume that we are given  $\lambda_1,\dots,\lambda_N\in P_k$, such
that $\mu=\sum_{i=1}^{N} \lambda_i$ is in the root lattice (otherwise $V^{\dagger}_{\vec{\lambda}}(\mathfrak{X})=0$). Write
$\mu=\sum_{p=1}^r n_p \alpha_p$, where $\alpha_p$ are the simple positive roots. It is easy to see that each $n_p$ is a non-negative integer.

Let $|\vec{\lambda}\rangle=|{\lambda}_1\rangle\tensor\dots\tensor
|{\lambda}_N\rangle$ be the tensor product of the corresponding highest weight vectors.
 Now (as in the introduction), consider and fix a  map $\beta:[M]=\{1,\dots,M\} \to R$, so that
$\mu=\sum_{a=1}^M \beta(a)$ with $M=\sum_{p=1}^r n_p$.

From the introduction, recall the variety $X_{\vec{z}}$ its cover $Y_{\vec{z}}$, its compactification $\overline{Y}_{\vec{z}}$
and the master function $\mathcal{R}$ on $Y_{\vec{z}}$. The following is the main result from ~\cite{B}:

Introduce variables $t_1,\dots,t_M$ considered as points on $\pone-\{\infty, P_1,\dots, P_N\}$. For every $\langle\Psi|\in
V^{\dagger}_{\vec{\lambda}}(\mathfrak{X})$, consider the correlation function (an explicit formula is recalled in Section \ref{shriek}).
\begin{equation}
\Omega=\Omega_{\beta}(\langle\Psi|) = \langle\Psi|f_{\beta(1)}(t_1)f_{\beta(2)}(t_2)\dots f_{\beta(M)}(t_M)|\vec{\lambda}\rangle.
\end{equation}
\begin{theorem}~\cite{TRR,B}
\begin{enumerate}
\item The multi-valued meromorphic form $\mathcal{R}\Omega$ on $X_{\vec{z}}$ is square integrable.
\item The differential form $\mathcal{R}q^*(\Omega)$ extends to an everywhere regular, single valued, differential form of the top order on any smooth and projective compactification $\overline{Y}_{\vec{z}}\supseteq Y_{\vec{z}}$.
\item The resulting map
\begin{equation}\label{class}
V^{\dagger}_{\vec{\lambda}}(\mathfrak{X}(\vec{z}))\hookrightarrow H^{M,0}(\overline{Y}_{\vec{z}},\Bbb{C})\subseteq H^M(\overline{Y}_{\vec{z}},\Bbb{C})
\end{equation}
 is injective.
\end{enumerate}
\end{theorem}

Note that by  ~\cite{SV}, the map $V^{\dagger}_{\vec{\lambda}}(\mathfrak{X}(\vec{z}))\hookrightarrow H^{0}(\overline{Y}_{\vec{z}},\Bbb{C})$
is flat for connections as $\vec{z}$ varies in the configuration space of $N$ distinct points on $\pone$ (with the KZ connection on $V^{\dagger}_{\vec{\lambda}}(\mathfrak{X}(\vec{z}))$ and the Gauss-Manin connection on $H^{0}(\overline{Y}_{\vec{z}},\Bbb{C})$).

\section{A review of logarithmic forms}\label{starlog}
 For a smooth algebraic variety $X$ of dimension $M$, there is  complex $\Omega^*_{\log}(X)$ of logarithmic forms on $X$. Let $j:X\to\overline{X}$ be a smooth compactification of $X$ such that the complement $D$ is a divisor with normal crossings. A regular differential $\omega\in H^0(X,\Omega_X^m)$ is said to be logarithmic if it lies in $H^0(\overline{X},\Omega_{\overline{X}}^m(\log D))$, this property does not depend upon the chosen compactification (see below). Recall that the complex $\Omega_{\overline{X}}^*(\log D)$ is the smallest subcomplex   of $j_*\Omega^*_X$
which contains $\Omega^*_{\overline{X}}$,  is stable under exterior products, and such that $df/f$ is a local section of $\Omega^1_{\overline{X}}(\log D)$ on an open subset $U$ whenever $f$ is meromorphic (algebraic) function on $\overline{X}$ which is regular on $X\cap U$.

Locally near a point of $D$ where $D$ is given by $z_1z_2\dots z_k=0$ and $z_1,\dots,z_M$ local coordinates on $\overline{X}$,  an element of $\Omega_{\overline{X}}^m(\log D)$ is a linear combination
$$\sum_I f_I d\eta_{i_1}\wedge\dots\wedge d\eta_{i_m},$$
where the sum is over subsets $I=\{i_1<\dots<i_m\}$ of $\{1,\dots,M\}$ of cardinality $m$, and $d\eta_i=dz_i/z_i$ if $i\leq k$ and $dz_i$ if $i>k$, and $f_I$ is a holomorphic function. Some basic properties are noted below.

\begin{enumerate}
\item $\Omega_{\overline{X}}^m(\log D)$ is a locally free sheaf on $\overline{X}$.
\item Log forms are suitably functorial: If $f:(\overline{X},D')\to (\overline{X},D)$ is a map of pairs as above, then there is an induced map
$\Omega_{\overline{X}}^m(\log D))\to f_*\Omega_{\overline{X}'}^m(\log D')$ and hence on the global sections,
$$H^0(\overline{X},\Omega_{\overline{X}}^m(\log D))\to H^0(\overline{X}',\Omega_{\overline{X}'}^m(\log D')).$$
\item Elements of $H^0(\overline{X},\Omega_{\overline{X}}^m(\log D))$ are $d$-closed for any $m$. The resulting map
\begin{equation}\label{mixedhodge}
H^0(\overline{X},\Omega_{\overline{X}}^m(\log D))\to H^m(X,\Bbb{C})
\end{equation}
is injective (corollaire 3.2.14 in ~\cite{D}).
\item The space of log-forms $\Omega^*_{\log}(X)$ on $X$ as a subspace of $H^0(X,\Omega^*_X) $is well defined, i.e., does not depend upon compactifications.
\end{enumerate}

\subsection{Complements of hyperplane arrangements}
We will restrict now to the case of $X=\Bbb{A}^M - S$ where $S=\cup_{i\in T} H_i$ is a hyperplane arrangement, where $H_i\subseteq \Bbb{A}^M$ is  given by linear equation $f_i=0$.
\begin{lemma}
The space of log forms on $X$ is the differential graded algebra over $\Bbb{C}$ inside the space of meromorphic differentials generated by the forms $df_i/f_i$.
\end{lemma}
\begin{proof}
Log forms of any degree embed in cohomology, so it suffices to show that the DG algebra over $\Bbb{C}$ inside the space of meromorphic differentials generated by the log forms $df_i/f_i$  maps surjectively into $H^*(X,\Bbb{C})$. This is proved in ~\cite{Brie}.
\end{proof}

\subsection{A criterion for log forms}\label{fatboy}

\begin{definition}
Let $Z$ be an $n$-dimensional smooth algebraic variety, and $\Gamma$ a possibly multi-valued $n$-form of the following form: For every $p\in Z$, there is an analytic open subset $U$ of $Z$ containing $p$, such that $\Gamma$ can be expressed as  $\Gamma=f\omega$ where
\begin{enumerate}
\item
$\omega$ is a (single valued) meromorphic form on $U$.
\item Some positive integer power of $f$ is a (single valued) meromorphic function on $U$.
\end{enumerate}
Let $S\subset Z$ be an irreducible subvariety. We will denote the logarithmic degree of $\Gamma$ along $S$ by $d^{S}(\Gamma)$. (See ~\cite{loo} for some  background on this concept). Briefly: Blow up  $Z$ along $S$, and let $E$ be the exceptional divisor. Then, $d^S(\Gamma)-1$ is the order of vanishing of (any branch of) $\Gamma$ along $E$.
\end{definition}

Consider $X=X_{\vec{z}}$ from the introduction. In a natural manner $X\subseteq (\pone)^M$. The complement $(\pone)^M-X$ is not a divisor with normal crossings, but is locally ``arrangementlike''.
Suppose $\Omega\in H^0(X,\Omega^M_X)$ is regular along the divisors $t_a=\infty$. The following gives a criterion to decide if $\Omega$ is a log form.

Consider the following types of strata $S\subseteq (\pone)^M$:
\begin{enumerate}
\item[(S1)] A certain subset of the $t$'s come together (to an arbitrary moving point). That is
$t_1=t_2=\dots= t_L$ after renumbering (possibly changing $\beta$).
\item[(S2)] A certain subset of the $t$'s come together to one of the $z$'s. That is
$t_1=t_2=\dots= t_L=z_1$ after renumbering (possibly changing $\beta$).
\end{enumerate}

See Section 10.8 in ~\cite{V}, and ~\cite{STV} for the proof of the following proposition.
\begin{proposition}\label{positive}
The following are equivalent:
\begin{enumerate}
\item $\Omega\in \Omega^M_{\log}(X)$.
\item The logarithmic degree of $\Omega$
along each stratum of type (S1), (S2) is $\geq 0$.
\end{enumerate}
\end{proposition}
\subsection{A basis for the space of log-forms}
Consider $X=X_{\vec{z}}$ from the introduction.
\begin{defi}\label{marker}
A marked partition of a finite set $A$ into $N$ parts is a pair $(\vec{\pi},\vec{k})$, where $\vec{k}=(k_1,\dots,k_N)$  is a sequence of non-negative integers such that  $\sum_{j=1}^N k_j=|A|$, and
$\vec{\pi}=(\pi_1,\dots,\pi_N)$ is a sequence of $N$ maps, with $\pi_j:[k_j]\to A$ such that
 \begin{enumerate}
 \item Each $\pi_j$ is injective.
 \item $A$ is the disjoint union of the images of $\pi_j$.
 \end{enumerate}
\end{defi}
To each marked partition of $[M]=\{1,\dots,M\}$ into $N$ parts, we assign  the differential
\begin{equation}\label{differ}
\Omega(\vec{\pi},\vec{k})\ =\  \eta_1\eta_2\dots\eta_N dt_1\wedge\dots\wedge dt_M,
\end{equation}
where $\eta_j=1$ if $k_j=0$ and
$$\eta_j=\frac{1}{(t_{\pi_j(1)}-t_{\pi_j(2)})(t_{\pi_j(2)}-t_{\pi_j(3)})\dots (t_{\pi_j(k_j)}-z_j)}, \text{ for } k_j>0.$$
Note that $\Omega(\vec{\pi},\vec{k})$ is (up to sign) the same as the wedge product
$\omega_1\wedge\omega_2\wedge\dots\wedge\omega_N$ where
$$\omega_j=d\log(t_{\pi_j(1)}-t_{\pi_j(2)}))\wedge d\log(t_{\pi_j(1)}-t_{\pi_j(2)}))\dots \wedge \dlog(t_{\pi_j(k_j)}-z_j),\text{ for } k_j>0,$$
and $\omega_j=1$ if $k_j=0$.
\begin{lemma}\label{todo}
The set of forms $\{\Omega(\vec{\pi},\vec{k})\}$, where  $(\vec{\pi},\vec{k})$ ranges over all marked partitions of $[M]$ with $N$ parts is a basis for the space $\Omega^M_{\log}(X_{\vec{z}})$ of top degree log forms on $X_{\vec{z}}$.
\end{lemma}
\begin{proof}
  Given a marked partition $(\vec{\pi},\vec{k})$
 of $[M]$ into $N$ parts, we can form a linear map $\Omega^M_{\log}(X_{\vec{z}})\to \Bbb{C}$ as the composition of the operators,  $$R_j=\Res_{t_{\pi_j(1)}=z_j} \Res_{t_{\pi_j(2)}=z_j}\dots\Res_{t_{\pi_j(1)}=z_j} \Res_{t_{\pi_j(k_j)}=z_j}$$ for $j=1,\dots,N$. Note that the operators $R_j$ commute, and by definition $R_j$ is the identity operator if $k_j=0$.
 The map corresponding to $(\vec{\pi},\vec{k})$ is non-zero on $\Omega(\vec{\pi},\vec{k})$ and vanishes on
  $\Omega(\vec{\pi'},\vec{k'})$ if $(\vec{\pi'},\vec{k'})\neq  (\vec{\pi},\vec{k}).$

 The forms  $\{\Omega(\vec{\pi},\vec{k})\}$ are linearly independent: Given a linear dependence relation, apply the map $\Omega^M_{\log}(X_{\vec{z}})\to \Bbb{C}$ corresponding to $(\vec{\pi},\vec{k})$. This shows that the coefficient of $\Omega(\vec{\pi},\vec{k})$ in the linear dependence relation is zero, as desired.

They span (cf. ~\cite{SV}): Consider a non-zero product of log forms $\eta$ of the form $\dlog(t_a-t_b)$ and $\dlog(t_a-z_j)$. Form a (undirected)  graph with $M+N$ vertices: the vertices are the  variables $t_a, a=1,\dots,M$ and $z_j, j=1,\dots,N$. The edges are the following : join   $t_a$ to $t_b$ if there is a term $\dlog(t_a-t_b)$ in $\eta$,  similarly join $t_a$ to $z_j$ if there is a $\dlog(t_a-z_j)$. There are no edges between a $z_i$ and a $z_j$.

Now note the following vanishing principle: if we have functions $f_1,\dots,f_s$ with  $\sum f_i=c$, $c$ a constant, then
$$d\log f_1\wedge d\log f_2\wedge\dots\wedge d\log f_s=\frac{1}{f_1 f_2\dots f_s}d f_1\wedge df_2\wedge\dots\wedge d f_s =0.$$
It is easy to see from this that the graph produced above does not have any cycles, and is hence a forest.
\begin{enumerate}
\item There is no path connecting a $z_i$ to a $z_j$. This is an immediate consequence of the vanishing principle above with $c=z_i-z_j$.
\item The union of connected components of $z_1,\dots,z_N$ is the entire graph: A connected component with vertices $t_{i_1},\dots,t_{i_{\ell}}$  produces a  $(\ell-1)$ differential. Since $\Omega$ is a differential of the top degree on $X_{\vec{z}}$, this is ruled out.
\item Look at the connected component $C$ of $z_1$, let $k_1+1=|C|$, for simplicity suppose that $C=\{t_1,\dots, t_{k_1},z_1\}$. We would like to trade the product of the differential forms corresponding to the edges of $C$ for a sum of terms (with coefficients) of the form
    \begin{equation}\label{turbine}
    \frac{1}{(t_{\pi(1)}-t_{\pi(2)})(t_{\pi(2)}-t_{\pi(3)})\dots (t_{\pi(k_1)}-z_1)}dt_1\wedge dt_2\wedge\dots\wedge dt_{k_1},
    \end{equation}
    where $\pi$ is a permutation of $[k_1]$.

  This can be achieved by using Lemmas 7.4.3 and 7.4.4 in ~\cite{SV}. Here we indicate an argument which uses the identity
  $$\frac{1}{(u-u_1)(u-u_2)}=\frac{1}{(u-u_1)(u_1-u_2)}+\frac{1}{(u-u_2)(u_2-u_1)}.$$

  Using this repeatedly we can assume (take $u=z_1$) that $z_1$ is connected to exactly one of the $t$'s, say $t_{1}$. We can then use this process to ensure that $t_{1}$ is connected to exactly one of the remaining $t$'s, say $t_{2}$ and continue. At the end, we will have
 a sum (with signs) of terms of the form \eqref{turbine}.

\end{enumerate}

\end{proof}
\subsection{Log forms stable under symmetries}\label{basislogforms}

Suppose we have $\beta$ as in the introduction. The group $S_M$ from the introduction acts on the space of log forms on $X_{\vec{z}}$. We will restrict this action to $\Sigma\subseteq S_M$, where $\Sigma=S_{n_1}\times \dots \times S_{n_r}$ is a product of symmetric groups as in the introduction.

Let $\epsilon$ denote the ``sign" of a permutation and $\Omega^M_{\log}(X_{\vec{z}})^{\epsilon,\Sigma}$ denote the $\epsilon$-character subspace, under the action of $\Sigma$ on $\Omega^M_{\log}(X_{\vec{z}})$. Consider  pairs $(\vec{\delta},\vec{k})$ where $\vec{k}=(k_1,\dots,k_N)$
with $\sum_{i=1}^N k_i=M$, $\vec{\delta}=(\delta_1,\dots,\delta_N)$ with $\delta_j:[k_j]\to R$ (not necessarily injective) and
\begin{equation}\label{DUH}
\sum_{j=1}^N\sum_{s=1}^{k_j}\delta_j(s) =\sum_{a=1}^M \beta(a)=\sum_{j=1}^N \lambda_j,
\end{equation}where $R=\{\alpha_1,\dots, \alpha_r\}$ is the set of positive simple roots of the Lie algebra $\frg$.
(Another way of saying this is that $\sum_{j=1}^N |\delta_j^{-1}(\alpha_p)|=n_p$ for each $p\in [r]$ in the notation from the introduction).
Denote the set of $(\vec{\delta},\vec{k})$ by $\mathcal{B}$.

For a element $(\vec{\delta},\vec{k})\in \mathcal{B}$ define a differential
\begin{equation}\label{cowboy}
\theta(\vec{\delta},\vec{k})=\sum _{{\pi}}\Omega(\vec{\pi},\vec{k}),
\end{equation}
where the sum is over all $\pi$ such that $(\vec{\pi},\vec{k})$ is a marked partition of $[M]$ with $N$ parts and with the constraint that $\beta\circ\pi_j=\delta_j$  for all $1\leq j \leq N$.

\begin{lemma}\label{chopin} The elements $\theta(\vec{\delta},\vec{k})$ for $(\vec{\delta},\vec{k})\in \mathcal{B}$ form a basis
of $\Omega^M_{\log}(X_{\vec{z}})^{\epsilon,\Sigma}$.
\end{lemma}
\begin{proof}
That they span follows from Lemmas \ref{todo} and \ref{stunt}: The $\theta(\vec{\delta},\vec{k})$ are  the $\chi$-averages of the basis $\Omega(\vec{\pi},\vec{k})$ of $\Omega^M_{\log}(X_{\vec{z}})$ with $G=\Sigma$. The linear independence is clear because a  basis vector $\Omega(\vec{\pi},\vec{k})$ appears in \eqref{cowboy} for a unique choice of $(\vec{\delta},\vec{k})$.
\end{proof}
\begin{lemma}\label{stunt}
Suppose a finite group $G$ acts on a finite dimensional complex vector space $V$. Let $\chi:G\to \Bbb{C}^*$ be a one dimensional character of
$G$. Then the  ``$\chi$-averaging" map $T:V\to V$ given by $Tv=\frac{1}{|G|}\sum \chi(g^{-1})gv$ is the projection to the $\chi$-isotypical subspace of $V$.
\end{lemma}

\section{From log forms to representation theory, first steps}
Suppose $R=\{\alpha_1,\dots,\alpha_r\}$ is the set of simple positive roots of $\frg$, and $e_1,\dots,e_r$ and $f_1,\dots, f_r$ be the corresponding elements in $\frg$.
Define a new Lie algebra: $\frg'$ is the Lie algebra with generators $e'_1,\dots,e'_r, f'_1,\dots, f'_r$ and $h\in \frh$ subject to the relations
$$[e'_i,f'_j]=\delta_{ij}h_i,$$
$$[h,e'_i]=\alpha_i(h) e'_i,[h,f'_i]=-\alpha_i(h)f'_i, [h,h']=0,$$
for all $1\leq i,j\leq r$; $h,h'\in\frh$.
Let
\begin{equation}\label{formnij}
n_{ij}=2\frac{(\alpha_i,\alpha_j)}{(\alpha_i,\alpha_i)}.
 \end{equation}
Consider elements  $$\theta_{ij}=\operatorname{ad}(e'_i)^{-n_{ij}+1}e'_j, \theta^-_{ij}=\operatorname{ad}(f'_i)^{-n_{ij}+1}f'_j.$$

There is a natural surjection $\frg'\to \frg$. Write
$$\frg=\mathfrak{n}^-\oplus \frh\oplus\mathfrak{n}, \frg'=\mathfrak{x}\oplus\frh\oplus\mathfrak{y}.$$
Let $\mathfrak{u}$ (resp. $\mathfrak{u}^-$) be the ideal of $\mathfrak{y}$ (resp. $\mathfrak{x}$)  generated by $\theta_{ij}$ (resp. $\theta^-_{ij}$). Then $\mathfrak{n}=\mathfrak{y}/\mathfrak{u}$, and
$\mathfrak{n}^-=\mathfrak{x}/\mathfrak{u}^-$ (see ~\cite{J.-P}).

\begin{defi}
For a dominant integral weight $\lambda$, let $M(\lambda)$ be the corresponding Verma module for the finite dimensional lie algebra $\frg$ with highest weight $|\lambda\rangle$. The corresponding Verma module for $\frg'$ will be denoted by $M'(\lambda)$. Let $V_{\lambda}$, a quotient of $M(\lambda)$, be the corresponding finite dimensional irreducible representation of $\frg$.
\end{defi}

See \cite{J.-P}  for the proof of the following:
\begin{lemma}\label{wash}
\begin{enumerate}
\item $M(\lambda)$ is a naturally isomorphic as an $\mathfrak{n}\oplus\frh$-module to the enveloping algebra $\mathcal{U}(\mathfrak{n}^-)$.
\item There is  an isomorphism of $\mathfrak{x}\oplus\frh$-modules $M'(\lambda)=\mathcal{U}(\mathfrak{x})$.  This isomorphism sends $1\in
\mathcal{U}(\mathfrak{x})$ to the highest weight vector in $M'(\lambda)$.
\item $\mathcal{U}(\mathfrak{x})$ is a free (complex) Lie algebra generated by $f_1',\dots,f_r'$.
\end{enumerate}
\end{lemma}

There is a surjection $M'(\lambda)\to M(\lambda)$, induced from the surjection  $\mathcal{U}(\mathfrak{x})\to \mathcal{U}(\mathfrak{n}^-)$. See ~\cite{J.-P} for the proof of the following proposition.
\begin{proposition}\label{language}
\begin{enumerate}
\item The kernel of $M'(\lambda)\to M(\lambda)$ is spanned by elements of the form $$f'_{i_1}\cdots f'_{i_k} \theta^-_{ij}f'_{j_1}\cdots f'_{j_l}|\lambda\rangle.$$
\item  The kernel $K(\lambda)$ of the natural surjection $M(\lambda)\to V_{\lambda}$ is generated as a $\frg$-module by the elements
$$f_i^{1+\frac{2(\lambda,\alpha_i)}{(\alpha_i,\alpha_i)}}|\lambda\rangle, i=1,\dots,r.$$
\end{enumerate}
\end{proposition}

Observe that for all $j$ (including $j=i$)
$$e_j f_i^{1+\frac{2(\lambda,\alpha_i)}{(\alpha_i,\alpha_i)}}|\lambda\rangle=0\in M(\lambda).$$
For $j\neq i$, this is clear because then $e_j$ commutes with $f_i$. For $j=i$, the computation reduces
to the case of $\mathfrak{sl}_2$. In this case one notes that if $h|\lambda\rangle=m|\lambda\rangle$, then
$ef^{m+1}|\lambda\rangle=(-m + (-m+2)+\dots+ (m-2) +m )|\lambda\rangle=0$.

Proposition ~\ref{language} implies that $K(\lambda)$ is spanned as a complex vector space by the elements of the form (where
$i$, $k$ and $i_1,\dots,i_k$ are arbitrary):
$$f_{i_1}\dots f_{i_k}f_i^{1+\frac{2(\lambda,\alpha_i)}{(\alpha_i,\alpha_i)}}|\lambda\rangle.$$

Writing $K'(\lambda)$ for the kernel of $\mathcal{U}(\mathfrak{x})=M'(\lambda)\to V_{\lambda}$, we see that it is a left $\mathcal{U}(\mathfrak{x})$-module spanned
(as a  $\Bbb{C}$-vector space) by elements of the form

$$f'_{i_1}\cdots f'_{i_k} \theta^-_{ij}f'_{j_1}\cdots f'_{j_l}|\lambda\rangle,\ i\neq j,\ i=1\dots,r,\  j=1,\dots r.$$
\begin{equation}\label{sound}
f'_{i_1}\cdots f'_{i_k}{f'}_i^{1+\frac{2(\lambda,\alpha_i)}{(\alpha_i,\alpha_i)}}|\lambda \rangle,\ i=1,\dots,r.
\end{equation}

\subsection{Tensor products}\label{notatio}
We now place ourselves in the setting of the introduction: $\lambda_1,\dots,\lambda_N$ are dominant integral weights, and $\sum_{a=1}^M\beta(a)=\sum \lambda_i=\sum_{p=1}^r n_p \alpha_p=\mu$.
Let $\widetilde{M}=M(\lambda_1)\tensor\dots\tensor M(\lambda_N)$ and $\widetilde{V}=V_{\lambda_1}\tensor\dots\tensor V_{\lambda_N}$. Similarly let $\widetilde{M}'=M'(\lambda_1)\tensor\dots\tensor M'(\lambda_N)$. There is a natural $\frh$-equivariant surjection $\widetilde{M'}\to \widetilde{M}$.

\begin{remark} The zero weight space of a $\frh$ module $T$ is denoted by $T_0$. For any $\frh$-module $W$ which is a direct sum of $\frh$-weight subspaces, there is a natural isomorphism
$$(W^*)_0\to (W_0)^*.$$
To see this note that an element of $(W^*)_{0}$ acts by zero on all elements of $W_{\nu}$ for all weights $\nu\neq 0$. For the reverse, use the direct sum decomposition $W=\oplus_{\nu}W_{\nu}$.
Note that $\widetilde{M},\widetilde{V}$ and $\widetilde{M}'$ are direct sums of $\frh$-weight subspaces.
\end{remark}

Consider $\widetilde{M'_0}$, the subspace of $\widetilde{M'}$ on which $\mathfrak{h}$ acts trivially (the zero weight space).
\begin{defi}\label{joyM}
For each $(\vec{\delta},\vec{k}) \in \mathcal{B}$ (defined in Section ~\ref{basislogforms}) associate the element $ w(\vec{\delta},\vec{k}):=|w_1\rangle\tensor |w_2\rangle\tensor\dots\tensor |w_N\rangle$
where $$|w_j\rangle= f'_{\delta_j(1)}f'_{\delta_j(2)}\dots f'_{\delta_j(k_j)}|\lambda_j\rangle\in M'(\lambda).$$
\end{defi}
\begin{lemma}
The vectors $w(\vec{\delta},\vec{k})$ for $(\vec{\delta},\vec{k})\in \mathcal{B}$ form a basis of $\widetilde{M'_0}$.
\end{lemma}
\begin{proof}
It follows from Lemma \ref{wash} that  $w(\vec{\delta},\vec{k})$ are linearly independent vectors. Now
 $\mathfrak{h}$ acts on $w(\vec{\delta},\vec{k})$ by weight $0$ (see equality \eqref{DUH}).
This shows that $w(\vec{\delta},\vec{k}) \in \widetilde{M'_0}$.

Now note the general fact: Suppose $V$ is a vector space (possibly infinite dimensional) with an action of $\frh$. Also assume that $V$ has a basis
consisting of eigenvectors for $\frh$. Then any zero weight vector (for $\frh$) in $V$ is a sum of elements of basis vectors which are of zero weight. Applying this to $\widetilde{M'}$ and basis vectors which are arbitrary tensors of vectors of the form  $f'_{\delta_j(1)} f'_{\delta_j(2)}\dots  f'_{\delta_j(k_j)}|\lambda_j\rangle$,
with $\delta_j:[k_j]\to R$ we see that $\widetilde{M'_0}$ is spanned by such vectors with
$$\sum_{j=1}^N\sum_{s=1}^{k_j} \delta_j(s) =\sum_{j=1}^N\lambda_j.$$
This shows that $(\vec{\delta},\vec{k}) \in \mathcal{B}$ (see equality \eqref{DUH}).
\end{proof}
The following is a result from ~\cite{SV}:
\begin{proposition}\label{SVProp}
There is a natural isomorphism
\begin{equation}\label{expli}
\Omega^{SV}_{\beta}:(\widetilde{M'_0})^*\to \Omega^M_{\log}(X_{\vec{z}})^{\epsilon,\Sigma},
\end{equation}
given by the formula (see Definition \ref{joyM})
\begin{equation}\label{free}
\Omega^{SV}_{\beta}(\langle\Psi|)=\sum_{(\vec{\delta},\vec{k})\in\mathcal{B}}\langle\Psi|{w}(\vec{\delta},\vec{k})\rangle \wedge\theta(\vec{\delta},\vec{k}).
\end{equation}

\end{proposition}
\begin{proof}
$\widetilde{M'_0}$ and $\Omega^M_{\log}(X_{\vec{z}})^{\epsilon,\Sigma}$ each have basis parameterized by $\mathcal{B}$ (see Section ~\ref{basislogforms}). The mapping $\Omega^{SV}_{\beta}$  sends the basis dual to the chosen basis of $\widetilde{M'_0}$ to the corresponding basis element of $\Omega^M_{\log}(X_{\vec{z}})^{\epsilon,\Sigma}$ and is hence an isomorphism.
\end{proof}
\begin{remark}
In ~\cite{SV}, Schechtman and Varchenko relate the Lie algebra homology of free Lie algebras to the cohomology (with local coefficients) of certain configuration spaces. The isomorphism
~\eqref{expli} is a particular case of their work (compare with $(7.2.4)$ and Section $7.1$ in \cite{SV}).
\end{remark}
\begin{remark}
It is easy to see that
\begin{equation}\label{morning}
\Omega^{SV}_{\beta}(\langle\Psi|)=\sum_{(\vec{\pi},\vec{k})}\langle\Psi|\vec{w}({\beta\circ\vec{\pi}},\vec{k})\rangle \wedge\Omega(\vec{\pi},\vec{k}),
\end{equation}
where $(\vec{\pi},\vec{k})$ vary over all marked partitions of $[M]$ into $N$ parts, and $\beta\circ\vec{\pi}\in\mathcal{B}$ is the element
$(\vec{\delta},\vec{k})$ with $\delta_j=\beta\circ\pi_j$ $j=1,\dots,N$.

\end{remark}

\subsection{Correlation functions and the Schechtman-Varchenko isomorphism}\label{shriek}
In the setting of Section \ref{marcel} (and using notation from Section \ref{notatio}), let $\langle\Psi|\in
V^{\dagger}_{\vec{\lambda}}(\mathfrak{X})$.

Note that $$V^{\dagger}_{\vec{\lambda}}(\mathfrak{X}(\vec{z}))\subseteq ((\widetilde{V})^*)^{\frg}\subseteq (\widetilde{M'_0})^*.$$ Therefore we can consider $\langle\Psi|$ as an element of $(\widetilde{M'_0})^*$ and apply the Schechtman-Varchenko morphism \eqref{free} to it. On the other hand, we have the correlation function $\Omega_{\beta}(\langle\Psi|)$ from Section \ref{marcel}. These coincide as will be shown below (Proposition \ref{DLH}):

\begin{equation}\label{machine}
\Omega^{SV}_{\beta}(\langle\Psi|) = \Omega_{\beta}(\langle\Psi|) =\langle\Psi|f_{\beta(1)}(t_1)f_{\beta(2)}(t_2)\dots f_{\beta(M)}(t_M)|\vec{\lambda}\rangle.
\end{equation}

The equality \eqref{machine} (see \eqref{expli}) should be viewed as an explicit formula for the correlation function
$$\langle\Psi|f_{\beta(1)}(t_1)f_{\beta(2)}(t_2)\dots f_{\beta(M)}(t_M)|\vec{\lambda}\rangle.$$

Suppose,
\begin{enumerate}
\item $|\vec{v}\rangle=|v_1\rangle\tensor\dots\tensor|v_N\rangle$ where each $|v_i\rangle \in V_{\lambda_i}$.
\item $X:A\subseteq[M]\to\mathfrak{n}^-$.
\item $\langle\Psi|\in
V^{\dagger}_{\vec{\lambda}}(\mathfrak{X})$.
\end{enumerate}
There is then an  explicit formula (cf. \cite{ATY}) for $$\langle{\Psi}|\prod_{a\in A} X_a(t_a)|\vec{\lambda}\rangle,$$ which generalizes \eqref{machine}. It is expressed as
a sum over marked partitions $(\vec{\pi},\vec{k})$, of $A$ into $N$ parts (see Definition \ref{marker}).
\begin{proposition}\label{DLH}
\begin{equation}\label{mav}
\langle{\Psi}|\prod_{a\in A} X_a(t_a)|\vec{\lambda}\rangle=\sum_{(\vec{\pi},\vec{k})} \frac{\langle\Psi|\bigl(\tensor_{j=1}^N X_{\pi_j(1)}X_{\pi_j(2)}\dots X_{\pi_j(k_j)} |v_j\rangle\bigr)}{\prod_{j=1}^N\bigl((t_{\pi_j(1)}-t_{\pi_j(2)})(t_{\pi_j(2)}-t_{\pi_j(3)})\dots (t_{\pi_j(k_j)}-z_j)\bigr)}\wedge_{a\in A} dt_a,
\end{equation}
where the sum runs through all marked partitions $(\vec{\pi},\vec{k})$ of $A$ into $N$ parts.\footnote{In the above expression $\bigl(\tensor_{j=1}^N X_{\pi_j(1)}X_{\pi_j(2)}\dots X_{\pi_j(k_j)} |v_j\rangle\bigr)\in V_{\lambda_1}\tensor V_{\lambda_2}\tensor\dots\tensor V_{\lambda_N}$, and the
product of differentials is taken in the order of $[M]$: if A$=\{a_1<\dots<a_s\}$ then $\wedge_{a\in A} dt_a=dt_{a_1}\wedge dt_{a_2}\wedge\dots\wedge dt_{a_s}$.}
\end{proposition}

\begin{proof}
The proof is by induction on $s=|A|$. Assume $A=\{t_1,\dots,t_s\}$ without any loss  of generality.
Let $\Theta=\langle{\Psi}| X_1(t_1)X_2(t_2)\dots X_s(t_s)|\vec{\lambda}\rangle$. If $s=1$, then the result is clear: start with
$\Omega=\langle{\Psi}|X_1(t)|\vec{\nu}\rangle$
Now use the function $\frac{1}{z-t}$ and the gauge condition (cf Page 70, ~\cite{Ueno}) to write
$$\Theta=\sum_{i=1}^N \frac{1}{t_1-z_i}\langle{\Psi}|\rho_i(X_1)|\vec{\nu}\rangle dt_1.$$

For $s>1$, let us write $\Theta=f_{\Theta}(t_1,\dots, t_s) d\vec{t}$ with $d\vec{t}=dt_1dt_2\dots dt_s$. We want to show that $f_{\Theta}$ equals the right hand side of  ~\eqref{mav} divided by $d\vec{t}$ (we do this to get rid of the   non-commuting $dt_1,\dots, dt_s$). We will show that both sides of the desired equality are equal as functions of $t_1$. It is easy to see that both sides vanish at infinity. We need to show that they have equal polar parts at every finite point. Therefore, we need to analyze the behavior as
\begin{enumerate}
\item $t_1 $ approaches $z_i$: Let $i=1$ for simplicity. The polar part of $\Theta$ is $\frac{1}{t_1-z_1}f_{\tilde{\Theta}}$ corresponding to a correlation function with variables $t_2,\dots, t_M$ (same $X$'s) with $|\nu_1\rangle$ changed to
$X_1 |\nu_1\rangle$. On the right hand side we need to consider only terms which have a fraction $\frac{1}{t_1-z_1}$. A little thought convinces us that
the equality of the polar parts at $t_1=z_1$ follows from induction.
\item $t_1$ approaches $t_a$. In this case the polar part  of $f_{\Theta}$  is $\frac{1}{t_1-t_a}f_{\tilde{\Theta}}$ corresponding to a correlation function  with points $t_2,\dots, t_M$, with $f_{\beta(a)}$ replaced by $[X_1,X_a]$. On the other side we should be looking at terms which have a $t_1-t_a$ or $t_a-t_1$ in the denominator. Firstly all partitions considered should have $t_1$ and $t_a$ in the same part. So we are looking at words which have $X_1X_a$ or $X_aX_1$ as sub words. We use the formula

    $$\alpha(t_1)\beta(t_a) X_1 X_a - \alpha(t_a) \beta(t_1)X_a X_1 =  \alpha(t_1)\beta(t_a)[X_1,X_a]+ O(t_1-t_a). $$
\end{enumerate}
\end{proof}

\section{The main theorems}
As stated in the introduction, our proof of Theorem ~\ref{main} is broken into two parts.
\begin{theorem}\label{Main1}
Assume $\frg$ is classical, or $G_2$.
Suppose  $\omega\in (H^{M,0}(\overline{Y}_{\vec{z}},\Bbb{C}))^\chi$. On ${Y}_{\vec{z}}$, express $\omega$ as a differential form $\mathcal{R}q^*\Omega$ where $q:Y_{\vec{z}}\to X_{\vec{z}}$ is the covering map. Then,
$\Omega$ is a log-form on $X_{\vec{z}}$.
\end{theorem}

From Theorem  ~\ref{Main1} and the Schechtman-Varchenko isomorphism ~\eqref{free} we can write any
$\omega\in (H^{M,0}(\overline{Y}_{\vec{z}},\Bbb{C}))^\chi$ in the form

$$\omega=\mathcal{R}q^*\Omega_{\beta}^{SV}(\langle\Psi|),$$
for some $\langle\Psi|\in (\widetilde{M_0'})^*$. Therefore Theorem ~\ref{main} will follow from the following:

\begin{theorem}\label{cook} Suppose $\langle\Psi|\in (\widetilde{M'_0})^*$ is such that $\mathcal{R}q^*\Omega_{\beta}^{SV}(\langle\Psi|)$   extends to (any) compactification of  $Y_{\vec{z}}$ (or equivalently that $\mathcal{R}\Omega_{\beta}^{SV}(\langle\Psi|)$ is a multivalued, square integrable form on $X_{\vec{z}}$). Then, $\langle\Psi|$ lies in the subspace $V^{\dagger}_{\vec{\lambda}}(\mathfrak{X}(\vec{z}))\subseteq \bigr( (V_{\lambda_1}\tensor\dots\tensor V_{\lambda_N})^*\bigr)^{\frg}\subseteq (\widetilde{M'_0})^*$ .
\end{theorem}
\begin{remark}
Note that $$\bigr( (V_{\lambda_1}\tensor\dots\tensor V_{\lambda_N})^*\bigr)^{\frg}\subseteq \bigr( (V_{\lambda_1}\tensor\dots\tensor V_{\lambda_N})^*\bigr)_{0}=((V_{\lambda_1}\tensor\dots\tensor V_{\lambda_N})_0)^* \subseteq (\widetilde{M}_0)^*\subseteq (\widetilde{M'_0})^*.$$
\end{remark}

We will prove Theorem ~\ref{cook} first, and return to the proof of Theorem ~\ref{Main1} in Section ~\ref{exam534}.


\begin{remark}\label{VarLoo}
Corollary 8.3 from ~\cite{loo} seems to imply Theorem ~\ref{Main1} immediately without restrictions on $\frg$. However,
we have not been able to follow the proof of this result from ~\cite{loo}. 
\end{remark}

\section{Proof of Theorem ~\ref{cook}}
Generalizing the considerations of Section \ref{shriek} we introduce more general ``correlation type'' functions: Suppose
\begin{enumerate}
\item $|\vec{v}\rangle=|v_1\rangle\tensor\dots\tensor|v_N\rangle$ where each $|v_i\rangle \in M'(\lambda_i)$.
\item $X:A\subseteq[M]\to \frg'$ with $X_a\in \mathfrak{x}$.
\item $\langle\Psi|\in (\widetilde{M'_0})^*$.
\end{enumerate}
\begin{definition}\label{aty}
Define $\langle\Psi|\prod_{a\in A} X_a(t_a)|\vec{v}\rangle$ to be
\begin{equation}
\sum_{(\vec{\pi},\vec{k})} \frac{\langle\Psi|\bigl(\tensor_{j=1}^N X_{\pi_j(1)}X_{\pi_j(2)}\dots X_{\pi_j(k_j)} |v_j\rangle\bigr)}{\prod_{j=1}^N\bigl((t_{\pi_j(1)}-t_{\pi_j(2)})(t_{\pi_j(2)}-t_{\pi_j(3)})\dots (t_{\pi_j(k_j)}-z_j)\bigr)}\wedge_{a\in A} dt_a,
\end{equation}
where the sum runs through all marked partitions $(\vec{\pi},\vec{k})$ of $A$ into $N$ parts as in Section \ref{shriek} (see Definition \ref{marker}).
\end{definition}

\begin{remark}
The authors do  not know if Definition \ref{aty} is a correlation function in conformal field theory (for non-integrable representations), also see equations (B4) and (B5) in ~\cite{ATY}.
\end{remark}
\begin{remark}\label{remarkable}
 Note that we can form a similar definition with $|v_i\rangle \in M(\lambda_i)$, $X:A\subseteq[M]\to \frg$ with $X_a\in \mathfrak{n}^-$, and  $\langle\Psi|\in (\widetilde{M_0})^*$ (these are objects for $\frg$ and not $\frg'$). These definitions are compatible: When $\langle\Psi|\in (\widetilde{M_0})^*$ and $X_a\in \mathfrak{x}$ one can project $X_a$ to $\mathfrak{n}^-$, and take the image of
  $\langle\Psi|$ in $(\widetilde{M'_0})^*$ and get two correlations functions, which coincide.
\end{remark}

\subsection{Residue formulas}\label{formulalist}
 If $\Omega$ is a meromorphic $M$-form on an algebraic variety $Y$ which presents at most simple poles along a smooth divisor $D$, define a meromorphic $(M-1)$ form $\Res_D \Omega$, on $D$ by
$$\Omega= \Omega'\wedge \frac{df}{f},$$
$$\Res_D \Omega=\Omega'|D,$$
 where $f$ is a local defining equation for $D$.
\begin{remark}
We need examine the  residues of $\Omega_{\beta}^{SV}(\langle\Psi|)$ to prove Theorem ~\ref{cook}. Since $\Omega_{\beta}^{SV}(\langle \Psi|)$ may not come from a correlation function, we need to work with the formula in Definition \ref{aty}.
\end{remark}
The locus $t_a=t_b$ in $\Bbb{C}^A$ will be parameterized by $\Bbb{C}^{A'}$ where $A'=A-\{a\}$ assuming $b<a$ (``keep the smaller variable''). The locus $t_a=z_j$, or $t_a=0$ (or $t_a=\infty$ in $(\pone)^A$) will be parameterized by $\Bbb{C}^{A'}$ (or $(\pone)^{A'}$) where $A'=A-\{a\}$. The form $\langle \Psi|\prod_{a\in A}X_a(t_a)|\vec{v}\rangle$ satisfies the following properties (recall that, by definition, $\langle\Psi|\in (\widetilde{M'_0})^*$):

\begin{enumerate}
\item $$\langle\Psi|\prod_{a\in A} X_a(t_a)|\vec{v}\rangle$$ is symmetric in $t_a$ and $t_b$ (up-to sign)
if $X_a=X_b$.

\item The residue of
$$\langle\Psi|\prod_{a\in A} X_a(t_a)|\vec{v}\rangle$$
along $t_a=t_b$ with $b<a$ is a similar function (up to sign)
$$\langle\Psi|\prod_{c\in A'} X'_c(t_c)|\vec{v}\rangle$$
with $A'=A-\{a\}$ and $X'_b=[X_a,X_b]$ (and $X'_c=X_c$ for $c\not\in \{a,b\}$).

\item The residue of
$$\langle\Psi|\prod_{a\in A} X_a(t_a)|\vec{v}\rangle$$
along $t_a=z_j$  is a similar function (up to sign)
$$\langle\Psi|\prod_{c\in A'} X'_c(t_c)|\vec{v'}\rangle$$
with $A'=A-\{a\}$ and $X'_c=X_c,\ \forall c\in A'$ and
$|v'_i\rangle=|v_i\rangle$ for $i\neq j$ and $|v'_j\rangle =X_a |v_j\rangle$.
\end{enumerate}
\begin{remark}\label{remarkable2}
There are similar formulas in the setting of Remark \ref{remarkable} (where one is looking at objects for $\frg$ rather than $\frg'$).
\end{remark}
\begin{remark}
There are similar properties for correlation functions in the theory of conformal blocks ~\cite{Ueno}.
\end{remark}

\subsection{Square integrability}

Suppose $\langle\Psi|\in (\widetilde{M'_0})^*$ is such that $\mathcal{R}q^*\Omega_{\beta}^{SV}(\langle\Psi|)$ is square integrable.
The first observation is that $\Omega=\Omega_{\beta}^{SV}(\langle \Psi|)$ is regular at the generic point of each of the divisors $t_a=\infty$. This is because
the degree of the function $\mathcal{R}$ along the divisor $t_a=\infty$ is negative ($=-(\beta(a),\beta(a))/\kappa$), therefore the logarithmic degree of $\Omega$ along this stratum is
$\geq 1$ which implies that $\Omega$ is regular along the divisors $t_a=\infty$.

We write $\Omega_{\beta}^{SV}(\langle \Psi|)$ as a ``correlation type'' function in the sense of Definition \ref{aty} (Prop \ref{SVProp} and Remark \ref{morning}):
\begin{equation}\label{quietus}
\Omega^{SV}_{\beta}(\langle\Psi|) = \langle\Psi|f'_{\beta(1)}(t_1)f'_{\beta(2)}(t_2)\dots f'_{\beta(M)}(t_M)|\vec{\lambda}\rangle.
\end{equation}
The formulas of the previous section allow us to consider suitable residues of the right hand side of \eqref{quietus} as ``correlation type''  functions.

Now we begin to probe the square integrability assumptions along deeper strata. For simplicity, back in the original situation assume that $\beta(2)=\beta(3)=\dots =\beta(-n_{ij}+2)=\alpha_i$
and $\beta(1)=\alpha_j$ (see \eqref{formnij}).

By Lemma ~\ref{observation} (5), we know that $\Res_{t_2=t_1}\Res_{t_3=t_1}\dots\Res_{t_{-n_{ij}+2}=t_1} \Omega=0$. This implies that
\begin{equation}\label{bourbaki}
\langle\Psi| \operatorname{ad}(f'_{\alpha_i})^{-n_{ij}+1}(f'_{\alpha_j})(t_1)\prod_{a>-n_{ij}+2}f'_{\beta(a)}(t_a)|\vec{\lambda}\rangle=0.
\end{equation}

\subsection{Proof of Theorem ~\ref{cook}, Part I}
Under the square integrability hypothesis we first prove $\langle\Psi|\in (\widetilde{M_0})^*\subseteq (\widetilde{M_0'})^*$.

The above formulas (formula ~\eqref{bourbaki} and Section~\ref{formulalist}) shows that $\langle\Psi|$ vanishes on any tensor $|w_1\rangle\tensor\dots\tensor|w_N\rangle$ where {\em some} $|w_a\rangle$
is of the form $$f'_{a_1}\dots f'_{a_k}\operatorname{ad}(f'_{\alpha_i})^{-n_{ij}+1}(f'_{\alpha_j})\dots f'_{a_{k+1}}\dots f'_{a_s}|\lambda_a\rangle$$

(we need consider only the case $|w_1\rangle\tensor\dots\tensor |w_N\rangle\in \widetilde{M_0}$, and we can use the description of $\widetilde{M},\widetilde{M'}$ in terms
of universal enveloping algebras, see Proposition ~\ref{language}).

\subsection{Proof of Theorem ~\ref{cook}, Part II}
We can now view $\Omega$ as a ``correlation-type''  object for $\frg$ via \eqref{quietus}, see Remark \ref{remarkable} and Remark \ref{remarkable2}.
Under the square integrability hypothesis we prove $\langle\Psi|\in (\widetilde{V_0})^*\subseteq (\widetilde{M_0})^*$, where $\widetilde{V}=V_{\lambda_1}\otimes \dots \otimes V_{\lambda_N}$.

Let $n=\frac{2(\lambda_j,\alpha_p)}{(\alpha_p,\alpha_p)}$
 and for simplicity, back in the original situation assume that $\beta(1)=\beta(2)=\dots =\beta(n+1)=\alpha_p$. By Lemma ~\ref{observation} (5), we know that $\Res_{t_1=z_j}\Res_{t_2=z_j}\dots\Res_{t_{n+1}=z_j} \Omega=0$. The rest of the argument is as in Part I (see Expression ~\eqref{sound}).
\subsection{Proof of Theorem ~\ref{cook}, Part III}

Under the square integrability assumption
$\langle\Psi|\in (\widetilde{V}^*)^{\frg}\subseteq (\widetilde{V_0})^*$:

We will now show that $f_i\langle\Psi|=0$ for all simple roots $\alpha_i$. To show this let $\beta(1)=\alpha_i$. Take residues at $t_2,\dots,t_M$ at $z_1,z_2,\dots, z_N$ (in all possible ways). One gets a differential
form in $t_1$ alone. The sum of its residues is zero (non-zero residues are possible only at $z_1,\dots,z_N$). This yields
$f_i\langle\Psi|=0$.

It follows that $\langle\Psi|\in (\widetilde{V}^*)^{\frg}$ (To show that $e_i \langle\Psi|=0$ for all $i$, we reduce to the case of $\mathfrak{sl}(2)$. It is then easy to see that the elements $e^m\langle\Psi|$, where $e=e_1$, generate a $\frg$-submodule of $(\widetilde{V}_0)^*$, all of whose weights are non-negative, the symmetry of weights forces these weights to be zero and hence $e\langle\Psi|=0$).

\subsection{Proof of Theorem ~\ref{cook}, Part IV}
It is  known that $\langle\Psi|\in (\widetilde{V}^*)^{\frg}$ lies in $V^{\dagger}_{\vec{\lambda}}(\mathfrak{X})$
if and only if
$$\langle\Psi| T^{k+1}|\vec v\rangle =0,\ \forall\ |\vec{v}\rangle\in V_{\lambda_1}\tensor\cdots \tensor V_{\lambda_N},$$
where
$$T:V_{\lambda_1}\tensor\cdots\tensor V_{\lambda_N}\to V_{\lambda_1}\tensor\cdots\tensor V_{\lambda_N}$$ is given by the formula
is the operator $\sum_{i=1}^N z_if^{(i)}_{\theta}$ with $f^{(i)}_{\theta}$ acting on the $i$th position of a tensor
product, see ~\cite{bea,FSV} (note that it is immaterial whether we choose $f_{\theta}$ or $e_{\theta}$.)

Suppose $\langle\Psi|\in (\widetilde{V}^*)^{\frg}$ is such that $\mathcal{R}\Omega$ is square integrable. We will now show that
$\langle\Psi|$ is actually in $V^{\dagger}_{\vec{\lambda}}(\mathfrak{X})$. Our task therefore, considering the previous paragraph, is to show that for any maps $\delta_j:[l_j]\to [r]$ for $j=1,\dots,N$, defining
\begin{equation}\label{taste}
|v_j\rangle= f'_{\delta_j(1)}\dots \tensor f'_{\delta_j(k_j)}|\lambda_j\rangle\in V_{\lambda},
\end{equation}
one has
\begin{equation}\label{apple}
\langle\Psi| T^{k+1}|v_1\rangle\tensor |v_2\rangle\tensor\dots\tensor |v_N\rangle =0.
\end{equation}
We note that ~\eqref{apple} is zero unless
$$\sum_{j=1}^N\sum_{\ell=1}^{l_j} \delta_{j}(\ell)=\sum_{j=1}^N \lambda_j - (k+1)\theta.$$
We will therefore assume that  $\sum \lambda_i- (k+1)\theta$ is a sum of simple positive roots.

By successively taking residues (if possible i.e., if $\sum \lambda_i- (k+1)\theta$ is $0$ or a  sum of positive simple roots) arrive at a correlation function
$$\langle\Psi|\prod_{a\in A} X_a(t_a)|\vec{\lambda}\rangle.$$
\begin{enumerate}
\item $X_1=X_2=\dots=X_{k+1}=f_{\theta}$ where $\theta$ is the highest root in $\frg$, and $X_{j}$ is in the weight space corresponding to negatives of simple roots for $j>k+1$).
\item  $|\vec{\lambda}\rangle=   |\lambda_1\rangle\tensor |\lambda_2\rangle\tensor\dots\tensor |\lambda_N\rangle$ where $|\lambda_j\rangle\in V_j$ is the highest weight vector.
\end{enumerate}

We claim that
$$\langle\Psi|\prod_{a\in A} X_a(t_a)|\vec{v}\rangle$$
has no poles when $t_a=t_b$ for $a,b\in [k+1], a\neq b$ and vanishes when
$t_1=\dots=t_{k+1}=\infty$. Let $\{t_a:a\in B\}$
be the set of variables which residuate to $t_1,\dots,t_{k+1}$.

The first part follows from $[f_{\theta},f_{\theta}]=0$ (so the residue at $t_a=t_b$ is zero). To prove the second part
 assume that $\langle\Psi|\prod_{a\in A} X_a(t_a)|\vec{v}\rangle$ does not vanish on the stratum $t_1=\dots=t_{k+1}=\infty$.
Then by Lemma ~\ref{grandres} (3), the logarithmic degree of
$\Omega$ on the stratum $S:t_b=\infty, \forall b\in B$ is $\leq k+1.$

The logarithmic degree of $\mathcal{R}\Omega$ is, on this stratum, (by Lemma \ref{grandres}, also look at calculations at infinity, $m=k+1$, see formulas from ~\cite{B} on stratum $(S3)$)
$$d^S(\mathcal{R}\Omega)\leq m +\deg_S(\mathcal{R})=m - \frac{m^2}{\kappa} -\frac{2m(g^*-1)}{2\kappa},$$
which is $\frac{m}{\kappa}$ times $\kappa -m -(g^*-1)= 0$, a contradiction to square integrability.
(See ~\cite{B}: equation (6.4), and the proof (there) of Lemma 6.1).

Now take appropriate residues of the  variables $t_j, j>k+1$ in $$\langle\Psi|\prod_{a\in A} X_a(t_a)|\vec{\lambda}\rangle$$
to arrive at a correlation function
$$\Omega'(t_1,\dots,t_{k+1})=\langle\Psi|\prod_{a\in A} X_a(t_a)|\vec{v}\rangle$$
where $A=\{1,\dots,k+1\}$ and $X_a=f_{\theta}$ for all $a\in A$ and
$|\vec{v}\rangle=   |v_1\rangle\tensor |v_2\rangle\tensor\dots\tensor |v_N\rangle$ where $|v_j\rangle\in V_j$ is as in ~\eqref{taste}. It is easy to see this new correlation function vanishes when $t_1=t_2=\dots=t_{k+1}=\infty$ (this requires a small argument in the style of Lemma ~\ref{lemmefondamental}).We will now show that the desired vanishing ~\eqref{apple} holds.

Note that  $\Omega'(t_1,\dots,t_{k+1})$ is a differential form with singularities only at $t_a=z_i$ and vanishes at $t_1=\dots=t_{k+1}=\infty$. The sum of residues in $t_1$ of the meromorphic form $t_1\Omega(t_1,\dots,t_{k+1})$ is zero. Its singularities are in the set $\{z_1,\dots,z_N,\infty\}$. Let $u_i=\frac{1}{t_i}$ to facilitate computations at infinity. Write $ \Omega'(t_1,\dots,t_{k+1})= f(t_1,\dots,t_{k+1})du_1\wedge du_2\wedge\dots\wedge du_{k+1}$.
We obtain
$$f(t_1,t_2,\dots,t_k,\infty)=-\sum_{i=1}^N \Res_{t_{k+1}=z_i}t_{k+1}\Omega'(t_1,\dots,t_{k+1}),$$
and iterating this, we obtain
$$0=f(\infty,\infty,\dots,\infty)=(-1)^{k+1}\prod_{a=1}^{k+1}\big(\sum_{i=1}^N \Res_{t_a=z_i}\big)t_1t_2\cdots {t_{k+1}}\Omega'(t_1,\dots,t_{k+1})$$
which immediately implies the desired equality ~\eqref{apple}.

\section{Lowest degree terms and logarithmic degrees along various strata}
\subsection{}\label{above} Let $Q(t_1,\dots,t_M)$ be a rational function in $t_1,\dots, t_M$ with   poles only along the diagonals of the form $t_a=t_b$ with $1\leq a <b \leq M$ and $t_a=z_j$, with $j=1,\dots,N$ (with $M$ arbitrary in this section) and let $S$ be the stratum $t_1=t_2=\cdots=t_L$.

We multiply $Q$ by an factor $\mathcal{P}=\prod_{1\leq a < b \leq L}(t_a-t_b)^{n_{a,b}}$ where $n_{a,b}\geq 0$, to get a rational function $\widetilde{Q}$, which is holomorphic (generically) on $S$. Let $u_a=t_a-t_1$ for $1< a \leq L$. We expand $\widetilde{Q}$ as power series with coefficients in the function field $K=K(S)$ of $S$.
$$\widetilde{Q}= \sum_{d \geq d_0}g_d(u_2, \cdots, u_L).$$

Note that we made a choice of a variable $t_1$ from the set $\{t_1,\dots,t_L\}$. Here $g_d$ is a homogeneous polynomial in the $u_a$'s with coefficients in $K(S)$ with total degree $d$ and  $d_0$ is the smallest number such that $g_{d_0}\neq 0$. Thus we can rewrite $Q$ as follows
\begin{equation}\label{jetlag}
Q =\frac{1}{\mathcal{P}} \sum_{d \geq d_0}g_d(u_2, \cdots, u_L).
\end{equation}
\begin{defi}
We refer to $\frac{g_{d_0}(u_2,\cdots, u_L)}{\mathcal{P}}$ as the lowest degree term of $Q$ and $d_0-\deg(\mathcal{P})$
as the degree of $Q$ on the stratum $S$. We also refer to $\mathcal{P}$ as a correction factor of $Q$ on the stratum $S$.
\end{defi}

\begin{remark}
Suppose $S$ is the stratum $t_1=\cdots=t_L=z_1$. We can repeat the above definitions of degree, lowest degree term and correction factors: We  multiply $\Omega$ by $\mathcal{P}=\prod_{1\leq a < b \leq L}(t_a-t_b)^{n_{a,b}}\prod_{1\leq a \leq L}(t_a-z_1)^{n_a}$ to get a function $\widetilde{Q}$ holomorphic on the generic point of $S$. We then expand $\widetilde{Q}$
in powers of $t_1-z_1, t_2-z_1,t_3-z_1,\dots, t_L-z_1$.
\end{remark}
\subsection{Some remarks on the lowest degree term}
The results of this subsection are not used elsewhere in this paper. In situation of Section \ref{above}, let
$h_d=g_d(t_2-t_1,\dots, t_L-t_1)$ for $d\geq d_0$. It is easy to see that $h_d$ is a polynomial  in $t_1,\dots, t_L$ with $K=K(S)$ coefficients. Note that $d_0$ and $h_d$ may (a priori) depend upon the choice of the ``initial variable'' $t_1$.

\begin{lemma}\label{polyvanish}The lowest degree and the lowest degree terms have the following properties:
\begin{enumerate}
 \item The lowest degree $d_0$, and the corresponding polynomial $h_{d_0}\in K[t_1,\dots,t_L]$ are independent of the choice of the initial variable $t_1$.

\item If $Q$ is symmetric in $t_1,t_2$, then so is $h_{d_0}$.

\item  Suppose $Q$ has no poles along $t_i=t_j$ for $i,j \in \{ 1,2, \dots, L'\}$ and vanishes on $t_1=\dots=t_{L'}$, then so does $h_{d_0}$ (here $L'\leq L$).
\end{enumerate}
\end{lemma}

\subsection{Logarithmic degree of meromorphic forms}
Let $\Omega$ be a top-degree meromorphic form on $\mathbb{A}^M$ such that $\Omega$ has poles only along the diagonals of the form $t_a=t_b$.
Write $\Omega=Q(t_1,\dots,t_M)\vec{dt}$, where $\vec{dt}=dt_1\dots dt_M$. Let $m$ be the degree of $Q$ on the stratum $S: t_1=\dots=t_L$. Then by an easy calculation,
\begin{lemma}
The logarithmic degree
$d^S(\Omega)$ of $\Omega$ along $S$ equals $m+L-1$.
\end{lemma}
We will call the lowest degree term of $Q$ on $S$ also as the lowest degree term of $\Omega$ on $S$.

\begin{lemma}\label{grandres}
 Suppose $\Omega$ has a simple pole along $t_1=t_2$. Let  $S$ and $S^*$ be the strata $t_1=t_2=\cdots=t_L$ and $t_1=t_3=t_4=\cdots=t_L$ respectively (the stratum $S^*$ is in variables $t_1,t_3,\dots,t_L$).
 Then,
 $$d^{S}(\Omega)\leq d^{S^*}(\Res_{t_2=t_1}\Omega).$$
In fact,
\begin{enumerate}
\item If the lowest degree term of $\Omega$ is holomorphic along $t_1=t_2$,
$$d^{S}(\Omega)<d^{S^*}(\Res_{t_2=t_1}\Omega).$$
\item If the lowest degree  term of $\Omega$  has a pole at $t_1=t_2$,
$$d^{S}(\Omega)=d^{S^*}(\Res_{t_2=t_1}\Omega).$$

\end{enumerate}
\end{lemma}

The following lemma shows that no new poles are created in the lowest degree term along any diagonal if we take residues along poles of the lowest degree term.
\begin{lemma}\label{division}
Suppose $\Omega$ has a simple pole along $t_2=t_3$. Further assume that the lowest degree term of $\Omega$ is holomorphic (generically) along $t_1=t_2$, $t_1=t_3$ and has a pole along $t_2=t_3$. Then the lowest degree term of $\Res_{t_3=t_2}\Omega$ is also holomorphic (generically) along $t_1=t_2$.
\end{lemma}
We end this section with a definition.
\begin{defi}Let $J=(j_1, j_2, \cdots, j_k)$ be an ordered subset of $[M]$ and $m$ be the minimum element in $J$. Let $K=([M]\backslash{J})\cup{\{m\}}$. We define $\Res_{\vec{J}}\Omega$ to be a form on $\mathbb{A}^{K}$ obtained from taking iterated residues of $\Omega$ along $t_m=t_a$  where $a \in J\backslash {\{m\}}$ following the order of the set $J \backslash\{m\}$ starting from the lowest.

\end{defi}
\section{The first step}\label{exam534}

Let $\Omega$ be any $M$-form on $X_{\vec{z}}$. As in Section \ref{fatboy}, we consider the following types of strata $S\subseteq (\pone)^M$:
\begin{enumerate}
\item[(S1)] A certain subset of the $t$'s come together (to an arbitrary moving point). That is
$t_1=t_2=\dots= t_L$ after renumbering (possibly changing $\beta$).
\item[(S2)] A certain subset of the $t$'s come together to one of the $z$'s. That is
$t_1=t_2=\dots= t_L=z_1$ after renumbering (possibly changing $\beta$).
\end{enumerate}
We note the following consequence of the square integrability assumption (cf.  ~\cite{loo}):
\begin{proposition}
Let $\mathcal{R}\Omega$ be a square integrable form on $X_{\vec{z}}$, then,
\begin{enumerate}
\item The logarithmic degree of $\mathcal{R}\Omega$ along a stratum $S:t_1=t_2=\dots=t_L$  $$d^{S}(\mathcal{R}\Omega):=d^{S}(\Omega)-\sum_{1\leq a <b \leq L}\frac{(\beta(a),\beta(b))}{\kappa}>0,$$ where $d^S(\Omega)$ is the logarithmic degree of $\Omega$ along $S$.
\item The logarithmic degree of $\mathcal{R}\Omega$ along a stratum $S:t_1=t_2=\dots=t_L=z_1$  $$d^{S}(\mathcal{R}\Omega):=d^{S}(\Omega)-\sum_{1\leq a <b \leq L}\frac{(\beta(a),\beta(b))}{\kappa} +\sum_{1\leq a \leq L}\frac{(\beta(a),\lambda_1)}{\kappa}>0,$$ where $d^S(\Omega)$ is the logarithmic degree of $\Omega$ along $S$.
\end{enumerate}

\end{proposition}
\begin{lemma}\label{observation}
Suppose $\mathcal{R}\Omega$ is square integrable.
\begin{enumerate}
\item If $a\neq b\in [M]$, then $\Omega$ has at most a simple pole along $t_a=t_b$.
\item If $(\beta(a),\beta(b))\geq 0$ for $a\neq b\in [M]$, then $\Omega$ does not have a pole along $t_a=t_b$.
\item $\Omega$ does not have a pole along $t_a=\infty$ for any $a\in [M]$.
\item $\Omega$ has at most a simple pole at $t_a=z_i$ for any $i$.
\item Suppose (after possibly changing $\beta$) that $\beta(1)=\dots=\beta(m)=\alpha$ and \newline $\widetilde{\Omega}=\prod_{a=1}^m (t_a-z_j)\Omega$. Then
$\widetilde{\Omega}$ vanishes at the  generic point of $t_1=\dots=t_m=z_j$  if $m\geq 1+ \frac{2(\lambda_j,\alpha)}{(\alpha,\alpha)}$.
\item Suppose (after possibly changing $\beta$) that  $\beta(2)=\beta(3)=\dots =\beta(m+1)=\alpha_i$
and $\beta(1)=\alpha_j$, and $\widetilde{\Omega}=\prod_{a=2}^{m+1} (t_{a}-t_{1})\Omega$. Then
$\widetilde{\Omega}$ vanishes at the generic point of $t_1=\dots=t_m=t_{m+1}$  if $m\geq 1- \frac{2(\alpha_j,\alpha_i)}{(\alpha_i,\alpha_i)} = 1 -n_{ij}$ (see \eqref{formnij}).
\end{enumerate}
\end{lemma}

\begin{proof}
Consider the stratum $t_a=t_b$. The logarithmic degree of $\Omega$  plus the quantity $\frac{-(\beta(a), \beta(b))}{\kappa}$ is positive. Therefore, $\Omega$ has a pole of order at most one along $t_a=t_b$, and if the poles are of order one then $(\beta(a), \beta(b))<0$. This gives us parts (1) and (2) of the lemma. The proof of (4) follows in the same way by considering the stratum $t_a=z_i$.

Since the degree of the function $\mathcal{R}$ along the divisor $t_a=\infty$ is negative ($=-(\beta(a),\beta(a))/\kappa$),  the logarithmic degree of $\Omega$ along this stratum is $\geq 1$ which implies that $\Omega$ is regular along the divisor $t_a=\infty$. This proves (3).

For (5) we consider the stratum $S$ defined by $t_1=\cdots =t_m=z_j$. The logarithmic degree $d^S(\mathcal{R}\Omega)$ along $S$ is positive. Thus we get the following:
\begin{eqnarray*}
d^S(\Omega) -\sum_{1\leq a < b \leq m}\frac{(\beta(a), \beta(b))}{\kappa} + \sum_{a=1}^m\frac{(\lambda_j, \beta(a))}{\kappa} & > & 0,\\
d^S(\Omega) - \frac{m(m-1)}{2}.\frac{(\alpha, \alpha)}{\kappa} + m\frac{(\lambda_j, \alpha)}{\kappa} &>& 0. \\
\end{eqnarray*}
If $m \geq 1 + \frac{2(\lambda_j, \alpha)}{(\alpha, \alpha)}$, then $d^S(\Omega) >0$ which implies that $\widetilde{\Omega}$ vanishes on the stratum $S$. The proof of (6) is similar to (5).
\end{proof}

\begin{remark}
To prove Theorem ~\ref{Main1}, we use Proposition ~\ref{positive}. For $\Omega$ as in the statement of Theorem ~\ref{Main1}, and each stratum $S$ of the form (S1) and (S2) we need to show that the logarithmic degree $d^S(\Omega)\geq 0$. The square integrability assumption gives $$d^S(\mathcal{R}\Omega)=d^S(\Omega)+\operatorname{deg_S}{\mathcal{R}}>0,$$ where $\operatorname{deg}_S(\mathcal{R})$ is the degree of $\mathcal{R}$ on a stratum $S$. So one may hope that $\operatorname{deg}_S\mathcal{R}$ on each stratum is $<1$. This may not be the case. For an example, let $\frg=\mathfrak{sl}_2$, let $\kappa$ be a large multiple of $4$,
let $m=\frac{\kappa}{4} +1$ and $\lambda=\frac{\kappa}{2}$. Then the $\operatorname{deg}_S\mathcal{R}$ on $S:t_1=\dots=t_M=z_1$ with
$\lambda_1=\lambda$ equals
$$\frac{1}{\kappa}(-m(m-1) + \lambda m)=m(\lambda-(m-1))=\frac{1}{4}(\frac{\kappa}{4}+1),$$
which is $>1$ (for large $\kappa$).
\end{remark}
\begin{remark}
Our argument uses the square-integrability information from a select set of strata to build a ``profile'' of $\Omega$ (Theorem ~\ref{controlpoles}), and then use this to prove that the logarithmic degree of $\Omega$ is non-negative on every stratum $S$.

\end{remark}

Assume now that $\frg$ is classical or $G_2$. We will prove the following property of the pole structure of $\Omega$. Let $T=\{1,2,\cdots, \ell\}$. Consider an iterated residue
$$\Omega'=\Res_{\vec{T}}\Omega=\Res_{t_{\ell}=t_1}\Res_{t_{\ell-1}=t_1}\dots\Res_{t_3=t_1}\Res_{t_2=t_1}\Omega.$$
Note that $\Omega'$ is a top degree form in $(t_1,t_{\ell+1},t_{\ell+2},\dots, t_M)$.
\begin{theorem} \label{controlpoles} Suppose $\mathcal{R}\Omega$ is square-integrable.
Assume that $\Omega'=\Res_{\vec{T}}\Omega\neq 0$. Then,
\begin{enumerate}
\item $\beta(1)+\dots+\beta(\ell)$ is a positive root.
\item The form $\Omega'$ has at most a simple pole along any of the divisors
$t_1=t_p, p>\ell$.
\item The form $\Omega'$  has at most simple pole along any of the divisors
$t_1=z_j$, for $j=1,\dots,N$.
\end{enumerate}
\end{theorem}
The following can be proved using Theorem ~\ref{controlpoles} and Lemma ~\ref{iter4}.
\begin{proposition}\label{controliter}
Let $I_1, \dots, I_n$ be pairwise disjoint subsets of $[M]$, with $I_j=\{a(j,1) <\dots < a(j,m_j)\}$, $|I_j|=m_j, \ j=1,\dots,n$. Then, the form $\Res_{\vec{I}_{n}}\cdots \Res_{\vec{I}_{1}}\Omega$ has at most simple poles along the sets $t_{a(j,1)}=z_i$ for $j=1,\dots,n$ and $i=1, \dots, N$.
\end{proposition}
\subsection{Proof of Theorem ~\ref{Main1}}
Given Theorem ~\ref{controlpoles} we will now prove Theorem ~\ref{Main1}. Let $\omega=\mathcal{R}\Omega$ be as in the statement of Theorem \ref{Main1}.  We need to show that the logarithmic degree of $\Omega$ along any stratum of the form $(S1)$ or $(S2)$ is non-negative. Our proof will follow a sequence of residues.
\begin{enumerate}
\item We always take residues along poles of a suitable lowest degree term of a form for a given stratum: in this case, by Lemma ~\ref{grandres}, the logarithmic degree does not change after taking residues.
\item The pole structure of   $\Res_{t_4=t_3}\Res_{t_2=t_1}\Omega$ as $t_3$ approaches some other variable can be controlled  by the pole structure of  $\Res_{t_4=t_3}\Omega$. The lemmas proved in Section \ref{IKEA} are crucial to this step.

\item If the lowest degree term of $\Omega$ for a given stratum $S$ defined by $t_1=t_2=\dots =t_L$ is holomorphic along $t_1=t_a$ for all $a \in \{2,\cdots, L\}$, then the lowest degree term of $\Res_{t_b=t_a}\Omega$ for the new stratum  $S^*$ remains holomorphic along $t_1=t_a$, where $S^*$ is  obtained by removing $t_b$ from the stratum $S$, and $a,b \in \{2,3,\dots, L\}$.

\end{enumerate}
We break up the proof into several steps. Let $S=S_1$ be a stratum of the form $t_1=\dots=t_L$ (a stratum of type $(S1)$).
\subsection{Step I}
Let $\beta(1)=\alpha_1$ and assume that $\Omega$ has a pole along $t_1=t_2$ and  $(t_1-t_2)$ does not divide the lowest degree term of $\Omega$ for the stratum $S$. We take a residue along $t_1=t_2$ to get a form $\Res_{t_1=t_2}\Omega$ and a new stratum $S_2$ defined by $t_1=t_3=\cdots =t_L$. By Lemma ~\ref{grandres}, $d^{S}(\Omega)=d^{S_2}(\Res_{t_1=t_2}\Omega)$. By Theorem \ref{controlpoles}, we know that $\Res_{t_2=t_1}\Omega$ has at most simple poles as $t_1$ approaches the remaining variables $t_a$ for $a=3,\dots, L$.

\subsection{Step II}We continue taking residues with the new form $\Res_{t_2=t_1}\Omega$ and the same variable $t_1$ along the stratum $S_2$. The simplicity of the poles of along $t_1=t_a$ where $t_a$ is any remaining variable is guaranteed by Theorem ~\ref{controlpoles}. When we cannot take any more residues,  we get a form $\Omega_k=\Res_{\vec{T}}\Omega$ and a stratum $S_k$, where $T$ denotes the ordered set of variables that got together during the residue process.

The lowest degree term of $\Res_{\vec{T}}\Omega$ for the stratum $S_k$ does not have a pole along $t_1=t_a$ where $a\in [L]\backslash T$. Also Lemma ~\ref{grandres} gives $d^{S}(\Omega)=d^{S_k}(\Res_{\vec{T}}\Omega)$. Let $b,c \in [L]\backslash T$, then the pole structure of $\Res_{\vec{T}}\Omega$ along $t_b=t_c$ is controlled by the poles structure of $\Omega$ along $t_b=t_c$ as in Theorem ~\ref{controlpoles}.

\subsection{Step III} We repeat Step I, Step II to the form $\Omega_k=\Res_{\vec{T}}\Omega$ and the stratum $S_k$ starting with a new variable. We keep taking residues along diagonals of the form $t_a=t_b$ unless all variables of all colors are exhausted. At the end we get a form $\Omega_n$ and a stratum $S_n$ such that the lowest degree term of $\Omega_n$ for the stratum $S_n$ is holomorphic. Then by Lemma ~\ref{grandres}, we get $d^{S}(\Omega)=d^{S_n}(\Omega_n)$. Thus $d^{S}(\Omega)\geq 0$.

 The proof that the logarithmic degree along any stratum of type $(S2)$: $t_1=\dots=t_L=z_1$ is non-negative follows similarly.
Note that we {\em do not} take residues along $t_k=z_1$. At the last step we will have set of surviving $t$ variables.
There are no poles in the lowest degree term when two of these variables are set together, and only (at most) a simple pole as one of them is  set equal to $z_1$ (Proposition ~\ref{controliter}). The logarithmic degree is easily seen to be non-negative.



\subsection{Some reductions in Theorem ~\ref{controlpoles}}
We will show that Theorem  ~\ref{controlpoles} reduces to the verification of Proposition \ref{redux}  below (under the assumption of square-integrability of $\mathcal{R}\Omega$).
Suppose $\Omega'=\Res_{\vec{T}}\Omega$ where $T=\{1,2,\dots, \ell\}$ and $p>\ell$.
\begin{proposition}\label{redux}
Assume $\beta(1)+\dots+\beta(\ell)$ is a positive root and $\beta(1)+\dots+\beta(\ell)+\beta(p)$ is not a {\em positive} root.
Then, $\Omega'=\Res_{\vec{T}}\Omega$ is regular along $t_1=t_p$.
\end{proposition}

\subsubsection{}\label{recap}
We will show that (3) of Theorem ~\ref{controlpoles} is immediate from (1) of Theorem ~\ref{controlpoles} and some  Lie algebra considerations.

Consider the stratum $S:t_1=t_2=\dots=t_L=z_1$. Our square-integrability assumption implies that $d^S(\mathcal{R}\Omega)> 0$, and $d^{S}(\Omega)\leq d^{S'}(\Res_{\vec{T}}\Omega)$ (by Lemma ~\ref{grandres}) with $S'$ the stratum $t_1=z_1$. Let $\gamma=\sum_{a=1}^{\ell}\beta(a)$.

Now
$$0<d^S(\mathcal{R}\Omega)= \frac{(\lambda_1,\gamma)}{\kappa}- \sum_{1\leq a<b\leq L} \frac{(\beta(a),\beta(b))}{\kappa} +d^{S}(\Omega).$$ We know that $\gamma$ is a positive root by (1) of Theorem ~\ref{controlpoles}. Using $(\lambda_1,\gamma)\leq k$ (since $\gamma$ is a root and $\lambda_1$ is of level $k$) and Lemma ~\ref{dualb} below, we see that from the above inequality, one gets
 $$0< \frac{k+g^*}{\kappa}+ d^{S}(\Omega)=1+ d^S(\Omega)\leq 1+ d^{S'}(\Res_{\vec{T}}\Omega).$$
 Therefore $d^{S'}(\Res_{\vec{T}}\Omega)>-1$, and this proves (3).
\begin{lemma}\label{dualb}
 Consider a positive root $\gamma=\sum_{i=1}^n{\delta_i}$, where $\delta_i$'s are positive simple roots (possibly repeated). Then,
$$\sum_{1\leq i<j\leq n} (\delta_i, \delta_j) > -g^*.$$
\end{lemma}
\begin{proof} Using Lemma 6.1 from \cite{B},
$$2\sum_{1\leq i<j\leq n} (\delta_i, \delta_j) =(\gamma,\gamma)-\sum_{i=1}^n (\delta_i,\delta_i)>(\gamma,\gamma)-2g^* > -2g^*.$$

\end{proof}
Suppose $\Omega'=\Res_{\vec{T}}\Omega$ is as in Theorem ~\ref{controlpoles}, where $T=\{1,2,\dots, \ell\}$ and $p>\ell$. The final step in showing that Theorem  ~\ref{controlpoles} reduces to the verification of Proposition \ref{redux}  is part (2) of the following:
\begin{proposition}\label{v8}Consider the stratum $S:t_1=\dots=t_{\ell}=t_p$.
\begin{enumerate}

\item If $d^S(\Omega)\geq 0$ (resp. $>0$), or equivalently the degree of $\Omega$ on $S$ is  $\geq -\ell$ (resp $>-\ell$), then $\Omega'$ has at most a simple pole (resp. holomorphic) along $t_1=t_p$.
\item If $\mathcal{R}\Omega$ is square integrable and $\beta(1)+\dots+\beta(\ell)+\beta(p)$ is a positive root, then $\Omega'$ has at most a simple pole along $t_1=t_p$.
\end{enumerate}
\end{proposition}
\begin{proof} Use the inequality $d^{S}(\Omega)\leq d^{S'}(\Res_{\vec{T}}\Omega)$ (by Lemma ~\ref{grandres})
where $S'$ is the stratum $t_1=t_p$. This shows (1).

For (2), we have $d^S(\mathcal{R}\Omega)>0$. Set $p=\ell+1$. Since $\sum_{a=1}^{\ell+1}\beta(a)$ is a positive root, by Lemma \ref{dualb},
$$0<d^S(\mathcal{R}\Omega)=d^{S}(\Omega) - \sum_{1\leq a<b\leq \ell+1} \frac{(\beta(a),\beta(b))}{\kappa} < d^S(\Omega)+ \frac{g^*}{\kappa}. $$
So $d^S(\Omega)> -\frac{g^*}{\kappa}$, and hence $d^S(\Omega)\geq 0$ and we can use (1).
\end{proof}

\subsection{Proposition \ref{redux} implies Theorem  ~\ref{controlpoles}} By the argument in Section \ref{recap}, we only need to deduce (1) and (2) of Theorem  ~\ref{controlpoles}
from Proposition \ref{redux}. Both are by induction on $\ell$, the base cases are covered by Lemma \ref{observation} ($\ell=1$). For the induction step for (1), we use
Proposition \ref{redux}: First we note that $\Res_{\vec{T}}\Omega\neq 0$, $T=\{1,2,\dots,\ell\}$, and hence by induction, $\sum_{i=1}^{\ell}\beta(i)$ is a simple root.  If $\sum_{i=1}^{\ell+1}\beta(i)$ is not a root, then by Proposition \ref{redux}, $\Res_{\vec{T}}\Omega$ is non-zero  is regular along $t_1=t_{\ell+1}$ and hence a further residue along $t_1=t_{\ell+1}$ produces zero.

For the induction step for (2), we divide into two cases. The first case is if  $\gamma=\sum_{i=1}^{\ell+1}\beta(i)+\beta(p)$ is a positive root, and handled using Proposition \ref{v8}. The second case is when $\gamma$ is not a positive root, which follows from Proposition \ref{redux}.

\subsection{}The proof of Proposition ~\ref{redux} is case by case. We will use the Bourbaki notation for Lie algebras.
\section{Proposition ~\ref{redux} for $\frg=\mathfrak{sl}(n+1)$}
\subsection{The case $\frg=\mathfrak{sl}(2)$} This case is immediate, because (by Lemma ~\ref{observation}) there are no poles for $\Omega$ at $t_a=t_b$.

\subsection{The case $\frg=\mathfrak{sl}(3)$}

Let $\alpha_1$, $\alpha_2$ (the non-simple root is $\alpha_1+\alpha_2$) denote the positive simple roots.
Suppose $\beta(1)=\alpha_1$ and $\beta(2)=\alpha_2$. Proposition \ref{redux} in this case follows from Lemma ~\ref{observation} and the following:
\begin{proposition}\label{madampsychosis}
The form $\Omega'=\Res_{t_2=t_1}\Omega$ in  $t_1, t_3, \dots,t_M$,  has no poles
as $t_1=t_a$ for any $a\in \{3,\dots,M\}$.
\end{proposition}
\begin{proof}
Suppose $\beta(3)=\alpha_2$. Then $\widetilde{\Omega}=(t_1-t_2)(t_1-t_3)\Omega$
is holomorphic at the generic point of $t_1=t_2$, and that of $t_1=t_3$ (also $t_2=t_3$); and vanishes
at $t_1=t_2=t_3$ by Lemma ~\ref{observation}. Therefore $\Omega'=\frac{\widetilde{\Omega}(t_1,t_1,t_3)}{(t_1-t_3)}$
which is regular at $t_1=t_3$. The proof when $\beta(3)=\alpha_1$ is similar.
\end{proof}
\subsection{The case $\frg=A_n=\mathfrak{sl}(n+1), n >2$}

We will follow the pattern of the case $\mathfrak{sl}(3)$. The simple roots are $\alpha_1,\dots,\alpha_n$ and the positive roots are of the form $\alpha_i + \alpha_{i+1} + \cdots  +\alpha_{j}$, where $i < j$.
We will have variables $t_1,\dots,t_M$ colored by the simple roots $\alpha_1,\dots,\alpha_n$. Suppose $t_1, \dots, t_{\ell}$ have colors $\alpha_1,\dots,\alpha_{\ell}$, for some $\ell\leq n$.
Then we want to prove the following:
\begin{proposition}\label{controlpolsl} The form
$$\Omega'=\Res_{t_{\ell}=t_{1}}\cdots\Res_{t_3=t_1}\Res_{t_2=t_1}\Omega $$ has poles along
$t_{1}=t_p$ only if $\ell\leq n-1$ and $\beta(p)=\alpha_{\ell+1}$. (In this case the pole is simple by Proposition
~\ref{v8}).
\end{proposition}

\begin{proof}
The proof is by induction on $\ell$, for $\ell=1$ (there are no residue operations), the statement is just
that $\Omega$ has poles along $t_1=t_p$ only if the color of $p$ is $\alpha_2$ ($(\alpha_1,\alpha_p)=0$ if
$p\geq 2$) which follows from Lemma ~\ref{observation}.
\subsubsection{The case $\ell=2$}
Let $\Omega'=\Res_{t_2=t_1}\Omega$. The proof that $\Omega'$ does not have a pole at $t_1=t_p$ if $\beta(p)\in \{\alpha_1,\alpha_2\}$ is similar to the $\mathfrak{sl}(3)$ case.

Let $\beta(p)\not\in \{\alpha_1, \alpha_2, \alpha_3\}$, then $\Omega$  is holomorphic at the generic point of $t_1=t_p$ and also at the generic point of $t_2=t_p$ ($\alpha_1+\beta(p)$ and $\alpha_2+\beta(p)$ are not roots). Thus by Lemma ~\ref{lemmefondamental}, we get $\Omega'$ does not have a pole on $t_1=t_p$.

\subsubsection{The case $\ell=3$}

Suppose $\beta(p) \not\in \{\alpha_2, \alpha_3, \alpha_4\}$, using the case $\ell=2$ we know that $\Res_{t_2=t_1}\Omega$ has no poles as $t_1=t_p$ and $t_3=t_p$. Hence by Lemma ~\ref{lemmefondamental}, $\Omega'$ is holomorphic along $t_1=t_p$.
\subsubsection{The case when $\ell=3$ and $\beta(p)=\alpha_2$}\label{argu1}
We will now show that $\Omega'$ is holomorphic along $t_1=t_p$ if $\beta(p)=\alpha_2$. Consider  $\widetilde{\Omega}=(t_1-t_2)(t_2-t_3)(t_1-t_p)(t_p-t_3)\Omega(t_1,t_2,t_3,t_p)$ which is regular at the generic point of $t_1=t_2=t_3=t_p$.

We consider $\widetilde{\Omega}(t_1,t,t_3,t)$. By Lemma ~\ref{observation}, $\widetilde{\Omega}$ vanishes when $t=t_1$ and $t=t_3$ and is hence divisible by $(t-t_1)(t-t_3)$ (i.e., the quotient is holomorphic at the generic point of $t=t_1=t_3$). One may
multiply by appropriate correction factors and view $\widetilde{\Omega}$ as a polynomial in the variables $t_i$.
At this point we view $\widetilde{\Omega}$ as a function (i.e., divide by $dt_1\dots dt_M$)

Next, look at $$\widetilde{\Omega}(t_1,t_2,t_3,t_p)-\widetilde{\Omega}(t_1, \frac{t_2+t_p}{2},t_3,\frac{t_2+t_p}{2})$$
which vanishes at $t_2=t_p$, is symmetric in $t_2,t_p$ and is hence divisible by $(t_2-t_p)^2$. We may therefore write
$$\widetilde{\Omega}(t_1,t_2,t_3,t_p)= (t_1-t_p)^2 (A) - (\frac{t_2+t_p}{2}-t_1)(\frac{t_2+t_p}{2}-t_3)(B).$$ The residue $\Omega'= \frac{\widetilde{\Omega}(t_1,t_1,t_1,t_p)}{(t_1-t_p)^2}$ and by the previous equation, the numerator vanishes at $t_1=t_p$ to the second order. Thus $\Omega'$ is holomorphic along $t_1=t_p$.

\subsubsection{The case when $\ell=3$ and $\beta(p)=\alpha_3$} We will show that $\Omega'$ does not have any poles as $t_1=t_p$  if $\beta(p)=\alpha_3$.
Consider
$\widetilde{\Omega}=(t_2-t_p)(t_2-t_3)(t_1-t_2)\Omega(t_1,t_2,t_3,t_p)$. Lemma ~\ref{observation} implies $\widetilde{\Omega}(t_1,t,t,t)=0$. Hence we can conclude that $\widetilde{\Omega}(t_1,t_2,t,t)$ is divisible by $(t-t_2)$. Now as before we look at
$$ \widetilde{\Omega}(t_1,t_2,t_3, t_p)- \widetilde{\Omega}(t_1, t_2, \frac{t_3+t_p}{2},\frac{t_3+t_p}{2}),$$ which vanishes at $t_3=t_p$ and is symmetric in $t_3$ and $t_p$. Thus we may write
$$\widetilde{\Omega}(t_1, t_2, t_3, t_p) = (t_p-t_3)^2(A) + (\frac{t_3+t_p}{2}-t_2)(B).$$ Thus $\Omega'=\frac{\widetilde{\Omega}(t_1,t_1,t_1,t_p)}{(t_1-t_p)}$ is holomorphic along $t_1=t_p$.

\subsubsection{The case of arbitrary $\ell$}By induction assume that the proposition is true for $\ell-1$.

Let $\Omega'$ be the iterated residue $\Res_{t_{\ell}=t_1}\cdots\Res_{t_2=t_1}\Omega$. Lemma ~\ref{lemmefondamental}  ensures that whenever $\beta(p)\neq \{ \alpha_{(\ell-1)}, \alpha_{\ell}, \alpha_{(\ell+1)}\}$, the form $\Omega'$ is holomorphic along $t_1=t_p$.

Consider the case when $\beta(p)\in\{\alpha_{(\ell-1)}, \alpha_{\ell}\}$. Using the same techniques as in case $\ell=3$ and $\beta(p)\in\{\alpha_2,\alpha_3\}$, we can show that $\Omega'$ is holomorphic along $t_1=t_p$.

\end{proof}
The proof of Proposition ~\ref{redux} for $\frg=\mathfrak{sl}(n+1)$ is now complete.
\begin{remark}
Our proof assumes that $t_1$ is colored by the simple root $\alpha_1$. The same argument works even if $t_1$ is colored by any $\alpha_i$, as long as all (subsequent) roots are to the ``right'' of $\alpha_i$. Since this is the case required for
our main argument, we will not write out the argument for the remaining cases.
\end{remark}

\begin{remark}
In every step of the proof of Proposition ~\ref{controlpolsl} we were reduced to checking two key things. We only needed to guarantee that at any stage the iterated residue is holomorphic along a variable which has the color of the last two roots added. All other cases were handled by Lemma ~\ref{lemmefondamental}. This reduction will also be used in the remaining cases.
\end{remark}
\section{Proposition ~\ref{redux} for $\frg=\operatorname{G}_2$}
The positive simple roots of $\operatorname{G}_2$ are $\alpha_1$ and $\alpha_2$. The other positive roots are $\alpha_1+\alpha_2$, $2\alpha_1+\alpha_2$, $3\alpha_1+\alpha_2$ and $3\alpha_1+ 2 \alpha_2$. The normalized Cartan Killing form is given by
$$(\alpha_1, \alpha_1)=\frac{2}{3},\  (\alpha_1, \alpha_2)=-1,\ (\alpha_2, \alpha_2)=2.$$

One can form the ``patterns'' of positive roots starting from $\alpha_1$, where at each step, one adds a simple root so that the sum is again a positive root. The only possible pattern is $\alpha_1,\alpha_1+\alpha_2$, $2\alpha_1+\alpha_2$,
$3\alpha_1+\alpha_2$ and $3\alpha_1+2\alpha_2$.

Let $\beta(1)=\alpha_1$. The cases $\ell\leq 2$ are easy and immediate.

\subsection{The case $\ell=3$} Assume $\beta(1)=\alpha_1$, $\beta(2)=\alpha_2$, $\beta(3)=\alpha_1$ and $\beta(p)=\alpha_2$. Let
 $\widetilde{\Omega}=(t_1-t_2)(t_1-t_p)(t_3-t_p)(t_3-t_2)\Omega$. Clearly
 $$\Omega'= \frac{\widetilde{\Omega}(t_1,t_1,t_1,t_p)}{(t_1-t_p)^2}.$$

Now $\widetilde{\Omega}$ is symmetric in $t_2$ and $t_p$ and vanishes at $t_2=t_p=t_1$ or at $t_2=t_p=t_3$. By the same argument as in Section ~\ref{argu1}, we can see that
$$\widetilde{\Omega}(t_1,t_2,t_3,t_p)= (t_2-t_p)^2 A + (\frac{t_2+t_p}{2}-t_1)(\frac{t_2+t_p}{2}-t_2) B,$$
and this shows that one can pull a $(t_1-t_p)^2$ out of $\widetilde{\Omega}(t_1,t_1,t_1,t_p)$, as desired.
\subsection{ The case $\ell =4$}
In this case by our previous arguments, $\beta(1)=\beta(3)=\beta(4)=\alpha_1$ and $\beta(2)=\alpha_2$.
Consider the form $\Omega'=\Res_{t_4=t_1}\Res_{t_3=t_1}\Res_{t_2=t_1}\Omega$. We will show that $\Omega'$ is holomorphic along $t_1=t_p$ if $\beta(p)=\alpha_1$.

We multiply $\Omega$ by a correction factor $\mathcal{P}=(t_1-t_2)(t_3-t_2)(t_4-t_2)(t_p-t_2)$ for the stratum $t_1=t_2=t_3=t_4=t_p$ and get a form $\widetilde{\Omega}$. By Lemma ~\ref{observation},  $\widetilde{\Omega}$ vanishes on $t_1=t_2=t_3=t_4=t_p$. Setting $w=\frac{t_1+t_3+t_4+t_p}{4}$, we see that $\widetilde{\Omega}(t_1, \cdots,t_4, t_p)-\widetilde{\Omega}(w, t_2, w, w, w)$ is symmetric in $t_1$, $t_3$, $t_4$, $t_p$ and vanishes on $t_1=t_3=t_4=t_p$.

By Lemma ~\ref{bansun}, we can rewrite $\widetilde{\Omega}(t_1, \dots, t_p)$ as a sum of terms of the form $(t_i-t_j)^2A_{ij}$ and $ \widetilde{\Omega}(w, t_2, w, w, w)$, where $i,j \in \{1,3,4, p\}$. Since $(w-t_2)$ divides $\widetilde{\Omega}(w, t_2, w, w,w)$, we get that $(t_1-t_p)$ divides $\widetilde{\Omega}(t_1, t_1, t_1, t_1, t_p)$. Hence $\Omega'$ has no poles at $t_1=t_p$ when  $\beta(p)=\alpha_1$.

\subsection{The case when $\ell=5$} Let $\beta(5)=\alpha_2$.
\subsubsection{The case when $\beta(p)=\alpha_1$.}The correction factor is $\mathcal{P}= (t_p-t_2)(t_p-t_5)(t_4-t_5)(t_3-t_5)(t_1-t_5)(t_4-t_2)(t_3-t_2)(t_1-t_2)\Omega$ for the stratum $t_1=t_2=t_3=t_4=t_5=t_p$. The form $\widetilde{\Omega}$ is symmetric in $t_1,t_3, t_4, t_p$ and in $t_2, t_5$. By Lemma ~\ref{observation},  $\widetilde{\Omega}$ vanishes along the partial diagonals of the form $t_1=t_3=t_4=t_p=t_2$(four of color $\alpha_1$ and one of color $\alpha_2$) and of the form $t_2=t_5=t_1$ (two of color $\alpha_2$ and 1 of color $\alpha_1)$. By Lemma ~\ref{mindeg2}
and Lemma  ~\ref{polyvanish}, $\widetilde{\Omega}$ has degree at least $4$  on the stratum $S:t_1=\dots=t_5=t_p$. Therefore the logarithmic degree of $\Omega$ on the stratum $S$ is at least $4-8+5>0$. Therefore, by Proposition ~\ref{v8}, $\Omega'$ is holomorphic along $t_1=t_p$.

\subsubsection{The case when $\beta(p)=\alpha_2$} We multiply $\Omega$ by a correction factor $\mathcal{P}$ for the stratum $S:t_1=\dots=t_5=t_p$ of degree $9$ to get a form $\widetilde{\Omega}$. The form $\widetilde{\Omega}$ vanishes on partial diagonals of the form $t_2=t_5=t_1$ (two of color $\alpha_2$ and one of color $\alpha_1$). By Lemma ~\ref{mindeg1} and Lemma ~\ref{polyvanish} the degree of $\widetilde{\Omega}$ on $S$ is at least $5$, and hence $d^S(\Omega)\geq 5 -9 +5 =1$. Using Proposition ~\ref{v8} we conclude that $\Omega'$ is holomorphic along $t_1=t_p$.

\section{The case $\frg=\operatorname{B}_n$} The positive simple roots are $\alpha_1, \dots, \alpha_n$. The positive roots of $\operatorname{B}_n$ are of the form $\alpha_i+ \alpha_{i+1}+\cdots+ \alpha_{n}$ for $1\leq i \leq n$; $(\alpha_i+\alpha_{i+1}+\cdots+ \alpha_n)+(\alpha_j+\alpha_{j+1}+\cdots +\alpha_n)$ for $1\leq i <j\leq n$; $\alpha_i+\alpha_{i+1}+\cdots+ \alpha_{j-1}$  for $1\leq i <j \leq n$. The highest root is $\theta=\alpha_1+2\alpha_2 + \cdots + 2\alpha_n$. The only possible ``pattern'' of positive roots starting at $\alpha_1$ is $\alpha_1, (\alpha_1+\alpha_2), \dots, (\alpha_1 + \dots +\alpha_n), (\alpha_1+ \dots +\alpha_{(n-1)}+ 2\alpha_n), \dots, (\alpha_1+ \alpha_2 + 2\alpha_3 + \dots + 2 \alpha_n), (\alpha_1+ 2\alpha_2+ \dots + 2\alpha_n)$.

The normalized Cartan killing form is given by
$$(\alpha_i, \alpha_i)=2,\ 1\leq i < n;\ (\alpha_n, \alpha_n)=1;\  (\alpha_i, \alpha_{i+1})=-1,\ 1\leq i < n,\ (\alpha_i, \alpha_j)=0,\ j> i+1.$$

Let $\beta(1)=\alpha_1$. We will now show Proposition ~\ref{redux} in this case.
We divide the proof into several cases. When $\ell=1$, there are no residues and by Lemma ~\ref{observation},  $\Omega$ has at most simple poles $t_{1}=t_{p}$  if $\beta(p)=\alpha_2$.

\subsection{The case $\ell \leq n$}
The proof in this case is similar to the proof of Proposition ~\ref{controlpolsl}.

\subsection{The case $n < \ell \leq 2n-2$.} Let $\beta(n+m)=\alpha_{(n-m+1)}$ for $m>0$. We prove the proposition in this case by induction on $\ell$.

\subsubsection{The initial step} When $\ell=n+1$ by Lemma ~\ref{lemmefondamental} and Proposition ~\ref{v8} we only need to consider the case when $\beta(p)=\alpha_{n}$. We multiply $\Omega$ by a correction factor $\mathcal{P}$ of degree $n+1$ for the stratum $t_1=\dots=t_{(n+1)}=t_p$ to get a form $\widetilde{\Omega}$. By Lemma ~\ref{observation} the form $\widetilde{\Omega}$ vanishes along $t_{(n-1)}=t_{n}=t_{(n+1)}=t_p$. Hence the logarithmic degree of $\Omega$ along $t_1=\dots=t_{(n+1)}=t_p$ is positive. The proof in this case is now complete by Proposition ~\ref{v8}.

\subsubsection{The inductive step} Let $\ell=n+m$ and $T=\{1,2, \dots, n+m\}$. Assume by induction  and Proposition ~\ref{v8} that for $m>1$, the meromorphic form $\Omega''=\Res_{\vec{T}}\Omega$ has at most simple poles along $t_1=t_p$ if $\beta(p)=\alpha_{(n-m)}$. We will show that the form $\Omega'=\Res_{t_{(n+m+1)}=t_1}\Omega''$ is holomorphic along $t_1=t_p$ if $\beta(p)\neq \alpha_{(n-m-1)}$. The proof is broken up into the following steps:

 Since $\Omega''$ has poles along $t_1=t_p$ if $\beta(p)=\alpha_{(n-m)}$. It is clear from Lemma ~\ref{lemmefondamental} that $\Omega'$ is holomorphic at a generic point of $t_1=t_p$ if $\beta(p)\not\in \{\alpha_{(n-m-1)}, \alpha_{(n-m)}, \alpha_{(n-m+1)}\}$. By Proposition ~\ref{v8} we know that $\Omega'$ has at most simple poles along $t_1=t_p$ if $\beta(p)=\alpha_{(n-m-1)}$.


Now consider the case when $\beta(p)=\alpha_{(n-m)}$. As before we multiply $\Omega$ by a correction factor $\mathcal{P}$ of degree $n+3(m+1)$ for the stratum $S$ defined by $t_1=t_2=\cdots =t_{(n+m+1)}=t_p$ to get a new form $\widetilde{\Omega}$. Lemma ~\ref{observation} tells us that the form $\widetilde{\Omega}$ satisfies the same property as that of the function $f$ in Lemma ~\ref{mindegk+3}. By Lemma ~\ref{mindegk+3} we see that the degree of $\widetilde{\Omega}$ for the stratum $S$ is at least $2m+3$ and hence $d^S(\Omega)>0$. Using Proposition ~\ref{v8}, we conclude that the form $\Omega'$ is holomorphic along $t_1=t_p$.

The case when $\beta(p)=\alpha_{(n-m+1)}$ is similar and follows from Lemma ~\ref{mindegk+4} and Proposition ~\ref{v8}.
\subsection{The case $\ell=2n-1$} Let $\beta(2n-1)=\alpha_{2}$. By Lemma ~\ref{lemmefondamental} we only need to check the cases when $\beta(p)\in\{ \alpha_1, \alpha_2, \alpha_3\}$. The proof  $\Omega'=\Res_{t_{(2n-1)}=t_1}\dots \Res_{t_2=t_1}\Omega$ is holomorphic along $t_1=t_p$ in these cases follow similarly using Lemma ~\ref{mindegk+1}, Lemma ~\ref{mindegk+3}, Lemma ~\ref{mindegk+4} and Proposition ~\ref{v8}.

The proof of Theorem \ref{controlpoles} for $\frg=\operatorname{B}_n$ is now complete.

\section{ Proposition ~\ref{redux} for $\frg=\operatorname{D_n}$} The positive simple roots of $\operatorname{D}_n$ are $\alpha_1, \alpha_2, \dots, \alpha_n$. The positive roots of $\operatorname{D}_n$ are of the form $\alpha_i +\alpha_{i+1}+\dots+ \alpha_{j-1}$ for $i<j<n$; $\alpha_i +\alpha_{i+1}+\dots+ \alpha_{j-1} + 2\alpha_{j} + 2 \alpha_{j+1}+ \dots + 2\alpha_{n-2} + \alpha_{n-1}+\alpha_n$ for $i<j < n-1$; $\alpha_i +\alpha_{i+1} + \cdots +\alpha_n $ for $i < n-1$; $\alpha_i + \alpha_{i+1}+ \dots+\alpha_{n-2}+ \alpha_n$ for $i \leq n-1$ and $\alpha_n$. If we formally put $\alpha_n=\alpha_{(n-1)}$ in the above expression of the positive roots we recover the positive roots of $B_{(n-1)}$.

There are two possible ``patterns" of positive roots starting at $\alpha_1$. The first pattern is $\alpha_1, (\alpha_1+\alpha_2), \dots, (\alpha_1+ \alpha_2 + \dots + \alpha_{(n-2)}), (\alpha_1 + \dots + \alpha_{(n-2)}+\alpha_{(n-1)}), (\alpha_1 + \dots + \alpha_{(n-2)}+\alpha_{(n-1)}+ \alpha_{n}), (\alpha_1 + \dots + \alpha_{(n-3)}+ 2\alpha_{(n-2)}+\alpha_{(n-1)} + \alpha_n), \dots, (\alpha_1 + 2\alpha_2 + 2 \alpha_3\dots + 2\alpha_{(n-3)}+ 2\alpha_{(n-2)}+\alpha_{(n-1)} + \alpha_n)$.

The second pattern is same as the first except the positive root $(\alpha_1 + \dots + \alpha_{(n-2)}+\alpha_{(n-1)})$ is replaced by the positive roots $(\alpha_1 + \dots + \alpha_{(n-2)}+\alpha_{n})$.

Since $\operatorname{D}_n$ is simply laced, the normalized Cartan killing form is given by
$$(\alpha_i, \alpha_i)=2,\ 1\leq i \leq n;\ \ (\alpha_{n-2},\alpha_n)=-1;\ (\alpha_i, \alpha_{i+1})=-1,\ i\leq n-2,(\alpha_{n-1},\alpha_n)=0,$$
$$(\alpha_i, \alpha_j)=0,\ i+1< j \text{ except when } i=n-2,\ j=n.$$

The proof of Theorem ~\ref{controlpoles} for $\frg=\operatorname{D}_n$ is same as the case $\operatorname{B}_n$. We only include the proof of Proposition ~\ref{redux} in the case $\frg=\operatorname{D}_4$.

Let $\beta(1)=\alpha_1$. When $\ell=1$, there is no residue and Lemma ~\ref{observation} tells us that $\Omega$ has at most simple poles at $t_1=t_p$ if $\beta(p)=\alpha_2$.
\subsection{The case $\ell=2$} We consider $\Omega'=\Res_{t_2=t_1}\Omega$. By Proposition ~\ref{v8} the form $\Omega'$ has at most simple poles along $t_1=t_p$ if $\beta(p)\in\{ \alpha_3, \alpha_4\}$. We will show that $\Omega'$ is holomorphic along $t_1=t_p$ if $\beta(p)\in\{\alpha_1, \alpha_2\}$. The proof in both these cases is similar to proof of Proposition ~\ref{controlpolsl}.

\subsection{The case $\ell=3$} We can assume that $\beta(3)=\alpha_3$. We consider the form $\Omega'=\Res_{t_3=t_1}\Res_{t_2=t_1}\Omega$. By Lemma ~\ref{lemmefondamental} and Proposition ~\ref{v8} we only need to show that $\Omega'$ is holomorphic along $t_1=t_p$ if $\beta(p)\in\{\alpha_2, \alpha_3\}$. The proof in this case also similar to the proof of Proposition ~\ref{controlpolsl}.

\subsection{The case $\ell=4$} Let $\beta(4)=\alpha_4$. We consider the form $\Omega'=\Res_{t_4=t_1}\dots\Res_{t_2=t_1}\Omega$. By Lemma ~\ref{lemmefondamental} and Proposition ~\ref{v8}, we only need to show that $\Omega'$ is holomorphic along $t_1=t_p$ if $\beta(p)=\alpha_4$.

\subsubsection{ The case $\beta(p)=\alpha_4$} We multiply the form $\Omega$ by a correction factor $\mathcal{P}=(t_1-t_2)(t_2-t_3)(t_2-t_4)(t_2-t_p)$ for the stratum $S$ defined by $t_1=t_2=t_3=t_4=t_p$ to get a form $\widetilde{\Omega}$. By Lemma ~\ref{observation}, the form $\widetilde{\Omega}$ vanishes on $t_2=t_4=t_p$. Thus $d^S(\Omega) \geq 1-4+4$. Hence the proof follows in this case by Proposition ~\ref{v8}.

\subsection{ The case $\ell=5$} Let $\beta(5)=\alpha_2$. We consider the form $\Omega'=\Res_{t_5=t_1}\dots\Res_{t_2=t_1}\Omega.$ We will show that $\Omega'$ is holomorphic along $t_1=t_p$.

\subsubsection{The case $\beta(p)=\alpha_2$} We multiply the form $\Omega$ by a multiplication factor of degree $9$ for the stratum $S$ defined by $t_1=t_2\dots=t_5=t_p$ to get a form $\widetilde{\Omega}$. By Lemma ~\ref{observation} the form $\widetilde{\Omega}$ satisfies the properties of the function $f$ in Lemma ~\ref{mindegk+3}. Hence by Lemma ~\ref{polyvanish} and Lemma ~\ref{mindegk+3} the degree of $\Omega'$ for the stratum $S$ is at least $5$. Now the proof follows from Proposition ~\ref{v8}.

\subsubsection{The case $\beta(p)=\alpha_1$} We multiply the form $\Omega$ by a multiplication factor of degree $8$ for the stratum $S$ defined by $t_1=t_2=\dots=t_5=t_p$. By Lemma ~\ref{observation} the form $\widetilde{\Omega}$ satisfies the properties of the function $f$ in Lemma ~\ref{mindeg3}. Hence the proof in this case follows as before.

\subsubsection{ The case $\beta(p)=\alpha_3$ or $\beta(p)=\alpha_4$} The proof that $\Omega'$ is holomorphic along $t_1=t_p$ is similar.

This completes the proof of Theorem ~\ref{controlpoles} for $\frg=\operatorname{D}_4$.

\begin{remark}
In the above proof for the case $\ell=3$, by assuming $\beta(3)=\alpha_3$ we followed the first pattern of the positive roots starting at $\alpha_1$. If we had assumed $\beta(3)=\alpha_4$, and followed the second pattern, the proof would have been similar.
\end{remark}
\section{Proposition ~\ref{redux} for $\frg=\operatorname{C}_n$}

The positive simple roots of $\operatorname{C}_n$ are $\alpha_1, \dots, \alpha_n$. The positive roots of $\operatorname{C}_n$ are of the form $\alpha_i+\alpha_{i+1}+\dots+ \alpha_j$ for $i<j \leq n$; $\alpha_i + \alpha_{i+1}+\dots+ 2\alpha_j + 2\alpha_{j+1}+\dots + 2\alpha_{n-1}+\alpha_n$ for $i\leq j<n$. The highest root $\theta$ is given by $2\alpha_1+2\alpha_2 + \dots+ 2\alpha_{n-1} + \alpha_n$. The only possible ``pattern" of positive roots starting at $\alpha_1$ is $\alpha_1, (\alpha_1+\alpha_2), \dots (\alpha_1+ \alpha_2+ \dots + \alpha_{n}), (\alpha_1+ \alpha_2 + \dots + \alpha_{(n-2)}+ 2\alpha_{(n-1)}+\alpha_n)+ (\alpha_1+ \alpha_2 \dots + \alpha_{(n-3)}+ 2\alpha_{(n-2)}+ 2\alpha_{(n-1)}+\alpha_n) + \dots (\alpha_1+ 2\alpha_2 \dots + 2\alpha_{(n-2)}+ 2\alpha_{(n-1)}+\alpha_n)+ (2\alpha_1+ 2\alpha_2+ \dots + 2\alpha_{(n-2)}+ 2\alpha_{(n-1)}+\alpha_n).$

The normalized Cartan killing form is given by the following $$(\alpha_i,\alpha_i)=1 \ \text{for $1\leq i <n$}; \ \ (\alpha_n,\alpha_n)=2;$$
$$(\alpha_i, \alpha_{i+1})=-\frac{1}{2} \ \text{for $1\leq i <n$};  \ \ (\alpha_{n-1}, \alpha_n)=-1 ;\ \
(\alpha_i, \alpha_j)=0 \  \text{for $i+1<j$}.$$

Let $\beta(1)=\alpha_1$. We give a proof Proposition ~\ref{redux} in this case by dividing the proof into several cases. When $\ell=1$, there are no residues and by Lemma ~\ref{observation}, $\Omega$ has at most simple poles $t_{1}=t_{p}$  if $\beta(p)=\alpha_2$.

\subsection{The case $\ell \leq  n$}
The proof follows easily from the same methods used in Proposition ~\ref{controlpolsl}.

\subsection{The case $\ell=n+1$} Let $\beta(n+1)=\alpha_{(n-1)}$ and $\Omega''=\Res_{t_n=t_1}\dots\Res_{t_2=t_1}\Omega$. The previous cases with $\ell\leq n$ and Proposition ~\ref{v8} tell us that the form $\Omega''$ has at most simple poles along $t_1=t_p$ if $\beta(p)=\alpha_{n-1}$. Thus by Lemma ~\ref{lemmefondamental}, we conclude that $\Omega'=\Res_{t_{n+1}=t_1}\Omega''$ is holomorphic along $t_1=t_p$ if $\beta(p)\not\in \{ \alpha_{(n-2)}, \alpha_{(n-1)}, \alpha_{(n)}\}$. If $\beta(p)=\alpha_{(n-2)}$, by Proposition ~\ref{v8},  $\Omega'$ has at most a simple pole along $t_1=t_p$.

\subsubsection{ The case when $\beta(p)=\alpha_n$} We multiply $\Omega$ with a correction factor of degree $n+3$ for the stratum $S$ defined by  $t_1=t_2=\dots=t_{(n+1)}=t_p$  to get a form $\widetilde{\Omega}$. By Proposition ~\ref{v8} we only need to show that $d^S(\Omega)>0$. Thus it is enough to show that the degree of $\widetilde{\Omega}$ for the stratum $S$ is at least $3$. By Lemma ~\ref{observation}, $\widetilde{\Omega}$  has the same properties as $f$ in Lemma ~\ref{deg3'}. Now the proof follows from Lemma ~\ref{deg3'} and Lemma ~\ref{polyvanish}.

\subsubsection{The case when $\beta(p)=\alpha_{(n-1)}$}When $\beta(p)=\alpha_{(n-1)}$ it follows similarly as above from Lemma ~\ref{deg3} that $\Omega'$ is holomorphic along $t_1=t_p$.


\subsection{ The case $n+1 \leq \ell < 2n-1$}Let $\beta(n-m)=\beta(n+m)=\alpha_{(n-m)}$ for $m <n$. We prove Proposition ~\ref{redux} by induction on $\ell$. The initial step $\ell=n+1$ is proved above. Now we prove the inductive step. We assume that for $1 \leq m-1$, the form $\Omega''=\Res_{t_{(n+m-1)}=t_1}\dots \Res_{t_2=t_1}\Omega$ is holomorphic along $t_1=t_p$ if $\beta(p)\neq \alpha_{n-m}$. By Proposition ~\ref{v8}, $\Omega''$ has at most simple poles along $t_1=t_p$ if $\beta(p)=\alpha_{(n-m)}$.

We consider the form $\Omega'=\Res_{t_{(n+m)}=t_1}\Omega''$. Lemma ~\ref{lemmefondamental} tells us that the form $\Omega'$ is holomorphic along $t_1=t_p$ if $\beta(p)\not\in \{ \alpha_{(n-m-1)}, \alpha_{(n-m)}, \alpha_{(n-m+1)}\}$. By Proposition ~\ref{v8}, $\Omega'$ has at most simple poles along $t_1=t_p$ if $\beta(p)=\alpha_{(n-m-1)}$. Thus we are reduced to check the following two cases:

\subsubsection{ The case when $\beta(p)=\alpha_{(n-m)}$} We multiply $\Omega$ by a correction factor $\mathcal{P}$  of degree $n+3m+1$ for the stratum $S$ defined by $t_1=t_2=\dots=t_{(n+m)}=t_p$ to get a form $\widetilde{\Omega}$. Lemma ~\ref{observation} tells us that the form $\widetilde{\Omega}$ has the same properties as the function $f$ in Lemma ~\ref{mig3}. Thus by Lemma ~\ref{polyvanish} and Proposition ~\ref{v8} we are done.

\subsubsection{ The case when $\beta(p)=\alpha_{(n-m+1)}$} We multiply $\Omega$ by a correction factor $\mathcal{P}$  of degree $n+3m+2$ for the stratum $S$ defined by $t_1=t_2=\dots=t_{(n+m)}=t_p$ to get a form $\widetilde{\Omega}$. By Lemma ~\ref{observation} the form $\widetilde{\Omega}$ has the same properties as the function $f$ in Lemma ~\ref{mig4}. Thus by Lemma ~\ref{polyvanish} and Proposition ~\ref{v8} we are done.

\subsection{The case $\ell=2n-1$} By Lemma ~\ref{lemmefondamental}, Proposition \ref{v8} and the previous step we only need to show that $\Omega'=\Res_{t_{(2n-1)}=t_1}\dots\Res_{t_2=t_1}\Omega$ is holomorphic along $t_1=t_p$ if $\beta(p)\in\{\alpha_1, \alpha_2\}$. The proof in this case is similar to the proof in the previous step.

The proof of Theorem \ref{controlpoles} for $\frg=\operatorname{C}_n$ is now complete.

\section{Key Lemmas}
Throughout this section $f$ will denote a polynomial in multiple variables which is symmetric in some variables and vanishes along certain partial diagonals. We use these properties of $f$ to give a lower bound on the total degree of $f$.
\subsection{For $\operatorname{G}_2$}

\begin{lemma}\label{bansun}
Suppose $g(u_1,\dots,u_n)$ is a symmetric polynomial in $u_1,\dots,u_n$ which vanishes on
$u_1=\dots=u_n$. Then $g$ is a linear combination of functions of the form $$(u_i-u_j)^2 {A}_{ij}(u_1,\dots,u_n),$$ where ${A}_{ij}$ are (possibly non-symmetric) polynomials.
\end{lemma}
\begin{proof}
We first show that $g$ is a linear combination of elements of the form $(u_i-u_j)B_{ij}$, where $B_{ij}$ are (possibly non-symmetric) polynomials. To see this, divide $g$ as a polynomial in $u_2$ with remainder, by $(u_2-u_1)$. The remainder is a (possibly non-symmetric) polynomial in $u_1,u_3,\dots,u_n$. Now continue with $(u_3-u_1)$ all the way until and including $(u_n-u_1)$. The final remainder
is a function in $u_1$ alone, which vanishes when $u_1=\dots=u_n$, and is hence zero.

Since $g$ is a symmetric polynomial, we just need to show that polynomials of the form
$$h=\sum_{\sigma\in S_n} \sigma\bigl((u_1-u_2)A(u_1,\dots,u_n)\bigr)$$
can be expressed as linear combinations of polynomials, each divisible by some $(u_i-u_j)^2$.
As $\sigma$ runs through $S_n$ so does $\sigma(12)$.
We can therefore rewrite the above sum as follows: $h$ equals
$$\frac{1}{2}(\sum_{\sigma\in S_n} \sigma\bigl((u_1-u_2)A(u_1,u_2,\dots,u_n)\bigr)+\sum_{\sigma\in S_n} \sigma \bigl((12)\bigl((u_1-u_2)A(u_1,\dots,u_n)\bigr)\bigr)=\frac{1}{2}\sum_{\sigma\in S_n} g_{\sigma},$$
where
$$g_{\sigma}= \sigma\bigl((u_1-u_2)(A(u_1,u_2,\dots,u_n)-A(u_2,u_1,\dots,u_n)\bigr),$$
i.e., the result of $\sigma$ acting on the polynomial $(u_1-u_2)(A(u_1,u_2,\dots,u_n)-A(u_2,u_1,\dots,u_n))$. Now
note that  $(u_1-u_2)$ divides $A(u_1,u_2,u_3,\dots,u_n)-A(u_2,u_1,u_3,\dots,u_n)$. Therefore
$g_{\sigma}$ is divisible by
$$\sigma\bigl((u_1-u_2)^2\bigr)=(u_{\sigma(1)}-u_{\sigma(2)})^2.$$

\end{proof}

\begin{lemma}\label{mindeg1}
Suppose that $f(u_1,u_2,u_3,t_1,t_2,t_3)$ is symmetric  in $u_1,u_2,u_3$ and vanishes on the nine diagonals of the form $u_1=u_2=t_1$ (two $u$'s and one $t$). Then, degree of $ f$ 
 is $\geq 5$.
\end{lemma}

\begin{proof}
Write $f$  as a sum
$$f= \bigr(f-f(\frac{\sum u_i}{3},\frac{\sum u_i}{3},\frac{\sum u_i}{3},t_1,t_2,t_3)\bigr) + h(\frac{\sum u_i}{3},t_1,t_2,t_3),$$
where $$h(w,t_1,t_2,t_3)=f(w,w,w,t_1,t_2,t_3)$$
vanishes when $w=t_1$ or $w=t_2$ or $w=t_3$, so we may write
$$h(w,t_1,t_2,t_3)= (w-t_1)(w-t_2)(w-t_3)g(w,t_1,t_2,t_3).$$

Clearly, $f-f(\frac{\sum u_i}{3},\frac{\sum u_i}{3},\frac{\sum u_i}{3},t_1,t_2,t_3)$ vanishes
on $u_1=u_2=u_3$, and is hence of the form
$$(u_1-u_2)^2 A + (u_2-u_3)^2 B+(u_1-u_3)^2 C.$$

Therefore,
\begin{equation}\label{chik}
f= (u_1-u_2)^2 A + (u_2-u_3)^2 B+(u_1-u_3)^2 C +(w-t_1)(w-t_2)(w-t_3)g(w,t_1,t_2,t_3),
\end{equation}
where $w=\frac{\sum u_i}{3}$. Put $u_1=u_2=t_1=c$, then the above equation reads
$$0=(u_3-c)^2 D+\frac{1}{3}(u_3-c)(\frac{2c+u_3}{3}-t_2)(\frac{2c+u_3}{3}-t_3)g(\frac{2c+u_3}{3},c,t_2,t_3).$$
Divide by $u_3-c$ and set $u_3=c$ to get $g(c,c,t_2,t_3)=0$, so $g(w,t_1,t_2,t_3)$ is divisible by
$(w-t_1)(w-t_2)(w-t_3)$.
The degree of $f$ is at least the degree of $h$, so if $h\neq 0$ we are done.

In the case $h=0$, consider ~\eqref{chik} with $u_1=u_2=c$. Note that
if $f$ vanishes on $u_1=u_2$ then it vanishes on all $3$ of the $u$ diagonals to order two and is hence of degree at least $6$. So we get the following:
$$f(c,c,u_3,t_1,t_2,t_3)= (c-u_3)^2 D(c,u_3, t_1,t_2,t_3), $$ where $D\neq 0$.

Since by hypothesis, $f$ vanishes of diagonal of the form $u_1=u_2=t_1$ (two $u$'s and one $t$) we get that the right hand side vanishes for $t_1=c$, so $D$ is divisible by $(c-t_1)(c-t_2)(c-t_3)$, and hence the
degree of $f$ is at least five.
\end{proof}

\begin{lemma}\label{mindeg2}
Suppose that $f(u_1,u_2,t_1,t_2,t_3,t_4)$ is symmetric separately in $u_1$, $u_2$ and $t_1$, $t_2$, $t_3$, $t_4$ and vanishes on diagonals of the form $u_1=u_2=t_1$ (two $u$'s and one $t$) and $t_1=t_2=t_3=t_4=u_1$ (four $t$'s and one $u$, two of these). Then, degree of $ f$
$\geq 4$.
\end{lemma}
\begin{proof}
Write $f$ as
$$(f-f(\frac{u_1+u_2}{2},\frac{u_1+u_2}{2},t_1,t_2,t_3,t_4)) + g(w,t_1,t_2,t_3,t_4),$$
where $g(w,t_1,t_2,t_3,t_4)=f(w,w,t_1,t_2,t_3,t_4)$ and $w=\frac{u_1+u_2}{2}$.

The term in the first bracket vanishes when $u_1=u_2$ and is symmetric in $u_1,u_2$.
Also note that $g(w,t_1,t_2,t_3,t_4)$ is of the form $(w-t_1)(w-t_2)(w-t_3)(w-t_4) h$.
So, we reduce to the case $g=0$ so
$$f(u_1,u_2,t_1,t_2,t_3,t_4)=(u_1-u_2)^2 A(u_1,u_2,t_1,t_2,t_3,t_4).$$
If $A$ vanishes on $t_1=t_2=t_3=t_4$, then by Lemma ~\ref{bansun}, it has degree $\geq 2$.
So assume that $A$ is non vanishing on $t_1=t_2=t_3=t_4=t$ and consider the following:
$$f(u_1,u_2,t,t,t,t)=(u_1-u_2)^2 A(u_1,u_2,t,t,t,t).$$
Put $u_1=t$ which makes the left hand side vanish and hence $A(u_1,u_2,t,t,t,t)$ vanishes when $u_1=t$.
Therefore $A(u_1,u_2,t,t,t,t)$ is divisible by $(u_1-t)(u_2-t)$ and we are done.
\end{proof}

\subsection{For $\operatorname{B}_n$ and $\operatorname{D_n}$}
\begin{definition}
A polynomial $f(x_1,\dots, x_n)$ is symmetric in the pair $(x_i,x_j)$ if $$\sigma_{i,j}(f(x_1,\dots,x_n))=f(x_1,\dots, x_n),$$ where $\sigma_{i,j}$ is the permutation $(i,j)$ in the symmetric group $S_n$.
\end{definition}
\begin{lemma}\label{mindeg3}
Suppose $f(t_1, u_1, t_2,u_2, t_3, u_3)$ is symmetric in the pairs $(t_1,u_1)$ and  $(t_2,u_2)$. Also assume that $f$ vanishes on diagonals of the form $t_a=u_a=t_{a+1}$ and $t_a=u_a=u_{a+1}$, for $a \in \{1,2\}$. Then, degree of $f\geq 4.$
\end{lemma}

\begin{proof}
Suppose $f$ vanishes on $t_1=u_1$. Then we can write $f$ as
$$f(t_1,u_1,t_2,u_2,t_3,u_3)=(t_1-u_1)^2f_1(t_1,u_1, t_2,u_2,t_3,u_3).$$ Consider the polynomial $g_1(t_2,u_2,t_3,v_3):=f_1(t_1,u_1,t_2,u_2,t_3,u_3)$. It follows from the properties of $f$ that $g_1$ is symmetric in $t_2,u_2$ and vanishes on the diagonals $t_2=u_2=t_3$ and $t_2=u_2=u_3$. Now it is easy to see that $g_1$ has degree at least two.

If $f$ does not vanish on $t_1=u_1$, we put $t_1=u_1=c$. Since $f$ vanishes on the diagonals $t_1=u_1=t_2$ and $t_1=u_1=u_2$, we get the following:
$$f(c,c,t_2,u_2,t_3,u_3)=(c-t_2)(c-u_2)f_2(c, t_2,u_2,t_3,u_3).$$
It is easy to see that the polynomial $g_2(t_2,u_2,t_3,u_3):=f_2(c,t_2,u_2,t_3,u_3)$  satisfies the same properties as $g_1$. Hence $g_2$ has degree at least two. This completes the proof.
\end{proof}
\begin{lemma}\label{mindegk+1}
Suppose $f(t_1,u_1, \cdots, t_{m+1},u_{m+1})$ is symmetric in the pairs $(t_a, u_a)$, for $1\leq a \leq m$. Further assume that $f$ vanishes on the diagonals of the form $t_a=u_a=t_{a+1}$ and $t_a=u_a=u_{a+1}$, where $1\leq a \leq m$. Then, degree of $f$ is at least $2m$.
\end{lemma}
\begin{proof}
We prove this lemma by induction on $m$. The case $m=1$ is easy and direct. Lemma ~\ref{mindeg3} is the case when $m=2$. Let us first consider the case when $f$ vanishes on $t_1=u_2$. Since $f$ is symmetric in $t_1,u_1$, we get the following:
$$f(t_1 , u_1, \cdots t_{m+1}, u_{m+1})=(t_1-u_1)^2f_1(t_1, u_1, t_2, u_2, \cdots, t_{m+1}, u_{m+1}).$$ A careful inspection shows that $g_1(t_2 , u_2, \cdots, t_{m+1}, u_{m+1}):=f_1(t_1, u_1, t_2, u_2, \cdots, t_{m+1}, u_{m+1})$ satisfies the same properties as function $f$ with $m$ variables. Hence by induction we are done.

If $f$ does not vanish on $t_1=u_1$, we put $t_1=u_1=c$. Since $f$ vanishes on the diagonal $t_1=u_1=t_2$ and $t_1=u_1=u_2$, we can write $f(c,c,t_2,u_2,\dots,t_{m+1},u_{m+1})$ as follows:
$$f(c,c,t_2,u_2,\dots, t_{m+1},u_{m+1})=(c-t_2)(c-u_2)f_2(c,t_2,u_2,\cdots, t_{m+1},u_{m+1}).$$

The polynomial $g_2(t_2,u_2,\dots, t_{m+1},u_{m+1}):=f_2(c,t_2,u_2,\cdots, t_{m+1},u_{m+1})$ satisfies the same properties as $f$ with $m$ variables. By induction, we get degree of $g_2$ is at least $2m-2$. Thus $f$ has degree at least $2m$.

\end{proof}
The following lemmas are proved by induction, Lemma ~\ref{mindeg1} and Lemma ~\ref{mindegk+1}.
\begin{lemma}\label{mindegk+3}
Suppose $f(w,x_1,x_2,x_3, t_1, u_1, t_2, u_2, \cdots, t_{m}, u_{m})$  is symmetric in $x_1, x_2, x_3$ and also symmetric in the pairs $(t_a,u_a)$, for all $1 \leq a < m$. Assume that $f$ vanishes on diagonals of the following form:
\begin{enumerate}
\item $x_1=x_2=w$ (two $x$'s and $w$).
\item $x_1=x_2=t_1$ (two $x$'s and $t_1$) and $x_1=x_2=u_1$ (two $x$'s and $u_1$).
\item $t_a=u_a=t_{a+1}$ and $t_a=u_a=u_{a+1}$, where $1\leq a <m$.
\end{enumerate}
Then, degree of $f$ is at least $2m+3$.

\end{lemma}

\begin{lemma}\label{mindegk+4}

Suppose $f(w, t_1, u_1, x_1, x_2,x_3, t_2, u_2, \cdots, t_{m}, u_{m})$  is symmetric in $x_1,x_2,x_3$ and also symmetric in the pairs $(t_a,u_a)$, for all $1 \leq a < {m}$. Assume that $f$ vanishes on diagonals of the following form:
\begin{enumerate}
\item  $x_1=x_2=t_1$ (two $x$'s and $t_1$) and $x_1=x_2=t_2$ (two $x$'s and $t_2$).
\item $x_1=x_2=u_1$ (two $x$'s and $u_1$) and $x_1=x_2=u_2$ (two $x$'s and $u_2$).
\item $t_1=u_1=w$ and $t_1=u_1=x_p$, where $p\in \{1,2, 3\}$.
\item  $t_a=u_a=t_{a+1}$ and $t_a=u_a=u_{a+1}$, where $2\leq a <m$.
\end{enumerate}
Then, $f$ has degree at least $2m+4$.
\end{lemma}

\subsection{ $\operatorname{C}_n$}
The proofs of the following lemmas are similar that of Lemma ~\ref{mindeg1} and Lemma ~\ref{mindeg3}.
\begin{lemma}\label{deg3}
 Assume $f(t_1, u_1, u_2, u_3, v_1)$ is symmetric in $u_i$'s and vanishes on diagonals of the form $u_1=u_2=t_1$ (two $u$'s and $t_1$ ) and $u_1=u_2=u_3=v_1$. Then, $f$ has degree at least $3$.

\end{lemma}
\begin{lemma}\label{deg3'}
Assume $f(t_1,u_1,u_2,v_1,v_2)$ is symmetric separately in $u$'s and $v$'s and vanishes on the diagonals of the form $u_1=u_2=t_1$ and $v_1=v_2=u_1$ (two $v$'s and one $u$). Then, $f$ has degree at least $3$.
\end{lemma}


The proofs of the next two lemmas are similar to the proofs of Lemma ~\ref{mindegk+1} and Lemma ~\ref{mindeg1}.

\begin{lemma}\label{mig3}

Suppose $f(w,x_1,x_2,x_3, t_1, u_1, t_2, u_2, \cdots, t_{m}, u_{m})$  is symmetric in $x_1, x_2, x_3$ and also symmetric in the pairs $(t_a,u_a)$, for all $1 \leq a \leq m$. Assume that $f$ vanishes on diagonals of the following form:
\begin{enumerate}
\item $x_1=x_2=w$ (two $x$'s and $w$).
\item $x_1=x_2=t_1$ (two $x$'s and $t_1$) and $x_1=x_2=u_1$ (two $x$'s and $u_1$).
\item $t_1=u_1=x_p$, where $p\in \{1,2,3\}$.
\item $t_a=u_a=t_{a-1}$ and $t_a=u_a=u_{a-1}$, where $1 < a \leq m$.
\end{enumerate}
Then, degree of $f$ is at least $2m+4$.
\end{lemma}

\begin{lemma}\label{mig4}
Suppose $f(w, t_1, u_1, x_1, x_2,x_3, t_2, u_2, \cdots, t_{m}, u_{m})$  is symmetric in $x_1,x_2,x_3$ and also symmetric in the pairs $(t_a,u_a)$, for all $1 \leq a \leq  {m}$. Assume that $f$ vanishes on diagonals of the following form:
\begin{enumerate}
\item  $x_1=x_2=t_1$ (two $x$'s and $t_1$) and $x_1=x_2=t_2$ (two $x$'s and $t_2$).
\item $x_1=x_2=u_1$ (two $x$'s and $u_1$) and $x_1=x_2=u_2$ (two $x$'s and $u_2$).
\item $t_1=u_1=w$ and $t_1=u_1=x_p$, where $p\in \{1,2, 3\}$.
\item $t_2=u_2=x_p$, where $p \in \{1,2,3\}$.
\item  $t_a=u_a=t_{a-1}$ and $t_a=u_a=u_{a-1}$, where $2 < a \leq m$.
\end{enumerate}
 Then, $f$ has degree at least $2m+5$.
\end{lemma}



\begin{remark}\label{f4remark}
In the case of $\frg=\operatorname{F}_4$ we will need to prove several degree lemmas, for example: Let $f$ be a polynomial in $12$ variables $u_1,u_2$, $v_1,v_2,v_3$, $w_1,w_2,w_3,w_4,w_5$, $x_1,x_2$, with the following properties
\begin{enumerate}
\item $f$ is symmetric (separately) in $u$'s, $v$'s, $w$'s and $x$'s.
\item $f$ vanishes on the following partial diagonals: equality of two $u$'s and one $v$, two $v$'s and one $u$, two $v$'s and one $w$, two $w$'s and one $x$, two $x$'s and one $w$; and three $w$'s and one $v$.
\end{enumerate}
Then, we will need to show that the degree of $f$ is $\geq 21$.
\end{remark}

\section{Properties of residues}\label{IKEA} Suppose $\Omega$ is a top degree form  defined in a neighborhood of $0\in \Bbb{C}^M$, which is  regular on the complement of the union of ($\binom{M}{2} +M$ many) divisors $t_i=t_j, i<j$ and $t_i=0$.

\begin{lemma}\label{lemmefondamental}
Suppose $\Omega$ has at most a simple pole along $t_1=t_2$.
\begin{enumerate}
\item Suppose  that $\Omega$ is regular at the generic point of $t_1=t_3$ and at the generic point of $t_2=t_3$. Then the form
$\Res_{t_1=t_2}\Omega$ in $(t_1,t_3,t_4,\dots)$ is generically regular on $t_1=t_3$.
\\
\item Suppose that $\Omega$ has a pole of order less than $n$ along $t_3=t_4$ and a simple pole along $t_1=t_2$. Then the form $\Res_{t_1=t_2}\Omega$ in $(t_1, t_3, t_4, \dots)$ has poles along $t_3=t_4$ of order less than $n$.
\end{enumerate}
\end{lemma}
\begin{proof}
For the first part we proceed as follows:

Assume
$$\Omega =\frac{g(t_1,\dots,t_M)}{P(t_1-t_2)}dt_1\wedge\dots\wedge dt_M,$$ where $g$ is a holomorphic function and $P$ is a polynomial in $t_1, \dots, t_M$ whose factors are of the form $t_i^{n_i}$ and $(t_i-t_j)^{a_{i,j}}$ for suitable exponents $n_i$ and $a_{i,j}$.

Since $P(t,t,t,t_4,\dots,t_M)\neq 0$ (generically), therefore the residue which is (up to sign)
$$\frac{g(t_1,t_1,t_3,\dots,t_M)}{P(t_1,t_1,t_3,\dots,t_M)}dt_1\wedge dt_2\wedge dt_4\wedge\cdots\wedge dt_M$$
is generically regular  on $t_1=t_3$.

The proof of (2) is similar to (1).
\end{proof}

\begin{remark} Note that if $P$ had a term of the form $(t_1+t_2-2t_3)$, then after $t_1=t_2$, we would have had a new pole at $t_1=t_3$ in the residue. It is important that the polar set of $\Omega$ does not contain sets like $t_1+t_3=t_2+t_4$, which after $t_1=t_2$ turn into a $t_3=t_4$.
\end{remark}

\begin{lemma}\label{iter4}
Let $I_1, I_2, \dots, I_n$ be pairwise disjoint subsets of $[M]$ such that $|I_j|=m_j$. Then for any $\sigma \in S_n$, the following equality of iterated residues of $\Omega$ holds:
$$\Res_{\vec{I}_{\sigma(1)}}\cdots \Res_{\vec{I}_{\sigma(n)}}\Omega=\Res_{\vec{I}_{1}}\cdots \Res_{\vec{I}_{n}}\Omega.$$
\end{lemma}

\bibliographystyle{plain}
\def\noopsort#1{}

\end{document}